\documentclass[12pt,reqno]{amsart}

\usepackage{amsmath,amsfonts,amssymb,amscd,amsthm,amstext,pgf,graphicx,hyperref,verbatim,lmodern,textcomp,color }
\usepackage{relsize}
\usepackage{csquotes}
\numberwithin{equation}{section}
\usepackage[mathscr]{euscript}
\theoremstyle{plain}
\newtheorem{thm}{Theorem}[section]
\newtheorem{lem}[thm]{Lemma}
\newtheorem{prop}[thm]{Proposition}
\newtheorem{cor}[thm]{Corollary}

\theoremstyle{definition}

\newtheorem*{thm*}{Theorem}

\theoremstyle{remark}
\newtheorem*{rem}{Remark}

\newcommand{\Ca}{\mathbb{C}^n}
\newcommand{\Ha}{\mathbb{H}^n}
\newcommand{\N}{\mathbb{N}}
\newcommand{\R}{\mathbb{R}}

\usepackage{color}
\newcommand{\blue}[1]{\textcolor{blue}{#1}}
\newcommand{\red}[1]{\textcolor{red}{#1}}

\newcommand{\C}{\mathbb C}%
\newcommand{\Bea}{\begin{eqnarray*}}
	\newcommand{\Eea}{\end{eqnarray*}}
\newcommand{\be}{\begin{equation}}
\newcommand{\ee}{\end{equation}}

\title[Ornstein-Uhlenbeck operator]
{An extension problem, trace Hardy and Hardy's inequalities for  Ornstein-Uhlenbeck  operator}

\author[Ganguly, Manna, Thangavelu]
{Pritam Ganguly, Ramesh Manna and  Sundaram Thangavelu}

\address{Department of Mathematics, Indian Institute of Science, 560 012 Bangalore, India}
\email{pritamg@iisc.ac.in, rameshmanna@iisc.ac.in, veluma@iisc.ac.in}
\date{}

\thanks{}

\begin{document}
\maketitle

\begin{abstract}
	In this paper, we study an extension problem for the Ornstein-Uhlenbeck operator $L=-\Delta+2x\cdot\nabla +n$ and we obtain various characterisations of the solution of the same. We use a particular solution of that extension problem to prove a trace Hardy inequality  for $L$   from which Hardy's inequality for fractional powers of $L$ is obtained.  We also prove an isometry property of the solution operator associated to the extension problem. Moreover, new $L^p-L^q$ estimates are obtained for the fractional powers of the Hermite operator.  
\end{abstract}

\tableofcontents
\section{Introduction and the main results}
 It is said that analysts are obsessed with inequalities. The usefulness of various weighted and unweighted inequalities in applications to problems in differential geometry, quantum mechanics, partial differential equations etc. have made this a very attractive area of research. Hardy's inequality is one such which   finds its origin in an old paper of G. H. Hardy \cite{H1} 
 written more than a hundred years ago (see also \cite{H2}).  In  recent years, this has been intensively studied in different settings and various contexts. For a historical review of Hardy's inequality, we refer the reader to the book of  Kufner et al. \cite{KMP}.
 
 We begin with recalling the classical Hardy's inequality which states that, given $f\in C^{\infty}_0(\R^n)$ we have 
 $$\frac{(n-2)^2}{4}\int_{\R^n}\frac{|f(x)|^2}{|x|^2}dx\leq \int_{\R^n}|\nabla f(x)|^2dx, \, ~n\geq 3$$ where $\nabla$ denotes the gradient in $\R^n$. This can be rephrased in terms of the Euclidean Laplacian $\Delta:=\sum_{j=1}^{n}\frac{\partial^2}{\partial x_j^2}$ as follows:
 $$\frac{(n-2)^2}{4}\int_{\R^n}\frac{|f(x)|^2}{|x|^2}dx\leq \langle (-\Delta)f,f\rangle $$ which has been generalised to fractional powers of the Laplacian. In fact, for $0<s<n/2$ and $f\in C^{\infty}_0(\R^n)$ the following holds
 \begin{equation}
 \label{h-homo}
 4^s\frac{\Gamma\left(\frac{n+2s}{4}\right)^2}{\Gamma\left(\frac{n-2s}{4}\right)^2}\int_{\R^n}\frac{|f(x)|^2}{|x|^{2s}}dx\leq \langle (-\Delta)^sf,f\rangle .  
 \end{equation}
  The constant appearing on the left hand side is known to be sharp (see e.g., \cite{B},\cite{Y}) but the equality is never achieved. In 2008, Frank et al. \cite{FLS}  used a ground state representation to give  a new proof of \ref{h-homo}, when $0<s<\min \{1,n/2 \}$ improving the previous results. On the other hand replacing the homogeneous weight $|x|^{-2s}$ by a non-homogeneous weight we have the following version of Hardy's inequality:
  \begin{equation}
  \label{h-nonhomo}
  4^s\frac{\Gamma\left(\frac{n+2s}{4}\right) }{\Gamma\left(\frac{n-2s}{4}\right) }\rho^{2s}\int_{\R^n}\frac{|f(x)|^2}{(\rho^2+|x|^2)^{2s}}dx\leq \langle (-\Delta)^sf,f\rangle, ~ \rho>0
  \end{equation}
  where the constant is sharp and the equality is achieved for the functions $(\rho^2+|x|^2)^{-(n-2s)/2}$. Note that proving such an inequality for fractional powers depends on how one views this type of operators. In fact, there are several ways of obtaining fractional powers of Laplacian. Caffarelli and Silvestre \cite{CS} first studied an extension problem associated to the Laplacian on $\R^n$ and obtained the fractional power as a mapping which takes Dirichlet data to the Neumann data. Motivated by this work, Boggarapu et al. \cite{BRT} studied the extension problem in a more general setting  of sum of squares of vector fields on certain stratified Lie groups. They used a solution of that extension problem to prove a trace Hardy inequality from which Hardy's inequality is obtained. Because of its several interesting features, the study of extension problem for various operators has received considerable attention in recent times, see e.g., \cite{RT3,ST} etc.
   
  Inspired by the work of Frank et al. \cite{FGMT}, Roncal-Thangavelu \cite{RT2} considered a modified extension problem for the sublaplacian on the H-type groups which gives conformally invariant fractional powers of the sublaplacian and they proved Hardy's inequality for the same. Although this inequality has been studied extensively  in the setting of Euclidean harmonic analysis, not  much has been studied in the frame work of Gaussian harmonic analysis.  As we   know  that the role of Laplacian in Gaussian harmonic analysis is played by the Ornstein-Uhlenbeck operator defined by $\tilde{L}:=-\Delta+ 2x \cdot\nabla $, it is therefore, natural to ask for such fractional Hardy  inequality for this operator. It is also convenient to work with $L:=-\Delta+  2x \cdot \nabla +n$ instead of $\tilde{L}$. Because of its various applications in probability theory, stochastic calculus etc.,  the study of Ornstein-Uhlenbeck operator experienced a lot of developments in the last couple of decades. We refer the reader to the   book of Urbina \cite{WU} in this regard. 
  
  \par  Our aim in this article is to establish Hardy and trace Hardy inequalities for fractional powers of the Ornstein-Uhlenbeck operator $ L $.  Recall that  $ L = -\Delta+2 x\cdot \nabla +n $ can be defined on the Gaussian $ L^2 $ space: $ L^2(\gamma)= L^2(\R^n, \gamma(x)dx), \gamma(x) = \pi^{-n/2}e^{-|x|^2} $ as a positive self adjoint operator.  We observe that $\sum_{j=1}^{n}\partial_{j}^*\partial_{j}=-\Delta+2x \cdot \nabla $  where $ \partial_j = \frac{\partial}{\partial x_j} $ and $ \partial_j^\ast = 2x_j- \partial_j $ is its adjoint on $ L^2(\gamma).$ The relation between $ L $ and the Hermite operator $ H = -\Delta+|x|^2 $ is given by $ M_\gamma L M_\gamma^{-1} = H $ where $ M_\gamma f(x) = \gamma(x)^{1/2} f(x).$  Hardy's inequality for the fractional powers  $ H^s $ of the Hermite operator has been studied in \cite{CRT}. Here $ H^s $ is defined by spectral theorem as  $$ H^s = \sum_{k=0}^\infty (2k+n)^s P_k$$
  	where $ (2k+n), k \in \N $ are the eigenvalues of $ H $ on $ L^2(\R^n) $ and $ P_k $ is the orthogonal projection of $ L^2(\R^n) $ onto the finite dimensional eigenspace corresponing to the eigenvalue $ (2k+n).$ However, there is another natural candidate for fractional powers of $ H $ and hence of $ L $ which will be treated here. 
  
     To motivate the new definition of fractional powers, denoted by $ H_s$ it is better to recall the conformally invariant fractional powers of the sublaplacian $ \mathcal{L} $ on the Heisenberg group $ \Ha$. The connection between $ \mathcal{L} $ and $ H$ is given by the relation $ \pi_\lambda(\mathcal{L}f) = \pi_\lambda(f) H(\lambda)$, where $ \pi_\lambda$ are the Schr\"odinger representations of $ \Ha $ and $ H(\lambda) = -\Delta+\lambda^2 |x|^2.$ The spectral decomposition of $ H(\lambda) $ is given by 
  	$$ H(\lambda) = \sum_{k=0}^\infty (2k+n)|\lambda| P_k(\lambda).$$ The conformally invariant fractional powers of $ \mathcal{L}$ are then defined, for $ 0 < s <  (n+1),$ by the relation
  	$$ \pi_\lambda(\mathcal{L}_s f) =  \pi_\lambda(f)  \sum_{k=0}^\infty (2|\lambda|)^s  \frac{\Gamma\left( \frac{2k+n+1+s}{2}\right)}{\Gamma\left( \frac{2k+n+1-s}{2}\right)} P_k(\lambda) .$$ 
  	The operator on the right hand side which multiplies $ \pi_\lambda(f) $ is the alternate candidate for fractional powers of $ H(\lambda) $ which we denote by $ H(\lambda)_s.$ By defining $ Q_k = M_\gamma^{-1} P_k M_\gamma$ the spectral decomposition of $ L $ becomes 
  	$  L = \sum_{k=0}^\infty (2k+n) Q_k $ and hence the fractional powers we are interested in are given by 
  	$$ L_s f(x)= \sum_{k=0}^\infty 2^s  \frac{\Gamma\left( \frac{2k+n+1+s}{2}\right)}{\Gamma\left( \frac{2k+n+1-s}{2}\right)} Q_k f(x) .$$ 
  	Along with $ L $ we also consider $ U = \frac{1}{2} L$ and the associated fractional powers 
  	$$ U_s f(x)= \sum_{k=0}^\infty 2^s  \frac{\Gamma\left( \frac{k+n/2+1+s}{2}\right)}{\Gamma\left( \frac{k+n/2+1-s}{2}\right)} Q_k f(x) .$$ For these operators, we prove the following inequality. Associated to $ A = L, U $ we define the trace norm of a function $ u(x,\rho) $ on $ \R^n \times [0,\infty) $ by
  	$$  a_s(A, u)^2 =  \int_{0}^{\infty}\int_{\mathbb{R}^n}\left(\left|\nabla_Au(x,\rho)\right|^2+\left(\frac{n}{2}+\frac{1}{4}\rho^2\right)u(x,\rho)^2\right)\rho^{1-2s}d\gamma(x)d\rho$$
  	where
  	$$ \nabla_{U}u:=\left( 2^{-1/2}\partial_1u, 2^{-1/2} \partial_2u,..., 2^{-1/2}\partial_nu, \partial_\rho u\right) $$
  	and $ \nabla_L $ is defined without the scaling factor $ 2^{-1/2}.$ 
  
  \begin{thm}[General trace Hardy inequality]
  	Let $0<s<1$ and let $ A $ be either $ L $ or $ U.$ Suppose $\phi\in L^2(\gamma)$ is a real valued function in the domain of $A_{s}$ such that $\phi^{-1} A_s\phi$ is locally integrable. Then for any    real valued   function $u(x,\rho)$ from the space $C^2_{0}\left([0,\infty), C^2_{b}(\mathbb{R}^n) \right)$   we have 
  	$$  a_s(A, u)^2 \ge 2^{1-2s}\frac{\Gamma(1-s)}{\Gamma(s)} \int_{\mathbb{R}^n}u(x,0)^2\frac{A_s\phi(x)}{\phi(x)}d\gamma(x).$$
  	
  \end{thm}
  It would be nice if we could choose a function $  \phi $ so that $ A_s \phi $ can be calculated explicitly. It turns out that for $ A = U $ we can do that. Indeed, with such a choice of $ \phi $ we can prove an explicit trace Hardy inequality from which a Hardy's inequality can be deduced. 
  
  \begin{thm}[Hardy's inequality for $U_s$]
  	\label{h-u_s}
  	Let $0<s<1$. Assume that $f\in L^2(\gamma)$ such that $U_{s}f\in L^2(\gamma)$.  Then   for every $ \rho > 0 $ we have  
  	\[\langle U_sf,f\rangle_{L^2(\gamma)} \ge (2\rho)^s \frac{\Gamma\left(\frac{n/2+1+s}{2}\right) }{\Gamma\left(\frac{n/2+1-s}{2}\right) } \int_{\R^n}\frac{f(x)^2}{(\rho+|x|^2)^s}
	 w_s(\rho+|x|^2)  \,\,d\gamma(x) \]
	  for an explicit  $ w_s(t) \geq 1.$ The inequality is sharp and equality is attained for $$ f(x) =C_{n,s}e^{\frac{|x|^2}{2}} (\rho+|x|^2)^{-(n/2+1+s)/2}K_{ (n/2+1+s)/2}(\rho+|x|^2)$$ 
	   where the constant $C_{n,s}$ is explicit, see \ref{phiexp}. 
  \end{thm}
We remark that  since    $ w_s(t) \geq 1$ we have the following inequality which is slightly weaker:
\begin{equation}
\label{hardyuuu}
\langle U_sf,f\rangle_{L^2(\gamma)} \ge (2\rho)^s \frac{\Gamma\left(\frac{n/2+1+s}{2}\right) }{\Gamma\left(\frac{n/2+1-s}{2}\right) } \int_{\R^n}\frac{f(x)^2}{(\rho+|x|^2)^s}
d\gamma(x).
\end{equation} 
    However, written in this form, we do not yet know if the constant appearing in the above inequality  is sharp or not. Observe that the constant we have obtained is analogous to the sharp constant in the Euclidean case (See \ref{h-nonhomo}).  
   It is worth pointing out that the Hardy's inequality for the pure fractional powers $U^s$ can be deduced from Theorem \ref{h-u_s}. Indeed, writing $R_s:=U_sU^{-s}$, we see that $R_s$ is a bounded operator on $L^2(\gamma)$ and its operator norm is given by 
   $$\|R_s\|_{op}=\sup_{k\geq 0} \left(\frac{k+n/2}{2}\right)^{-s} \frac{\Gamma\left( \frac{k+n/2+1+s}{2}\right)}{\Gamma\left( \frac{k+n/2+1-s}{2}\right)}.$$ To estimate this norm we use the fact that for $\alpha>0$, $x^{\beta-\alpha}\frac{\Gamma(x+\alpha)}{\Gamma(x+\beta)}\leq \frac{x+\beta}{x+\alpha}$ (see \cite{RT1})  which gives the following estimate
   $$\left(\frac{k+n/2}{2}\right)^{-s} \frac{\Gamma\left( \frac{k+n/2+1+s}{2}\right)}{\Gamma\left( \frac{k+n/2+1-s}{2}\right)}\leq \frac{2k+n+2(1-s)}{2k+n+2(1+s)}.$$ The right hand side of the above inequality being an increasing function of $k$, we obtain $\|R_s\|_{op}\leq 1.$ 
      Using this Hardy's inequality for $U^s$ reads as 
   \begin{cor}
   	\label{h-u^s}
   	Let $0<s<1$. Assume that $f\in L^2(\gamma)$ such that $U^sf\in L^2(\gamma)$.  Then for any $\rho>0$ we have 
   	\[ \langle U^sf,f\rangle_{L^2(\gamma)} \ge (2\rho)^s \frac{\Gamma\left(\frac{n/2+1+s}{2}\right) }{\Gamma\left(\frac{n/2+1-s}{2}\right) } \int_{\R^n}\frac{f(x)^2}{(\rho+|x|^2)^s}d\gamma(x). \]
   \end{cor}
 As consequences of Hardy's inequality with non-homogeneous weight, we obtain a Heisenberg type uncertainty principle for the fractional powers of the Ornstein-Uhlenbeck operator. Indeed,  an application of Cauchy-Schwarz inequality yields 
  $$\int_{\R^n}|f(x)|^2d\gamma(x)\leq \left(\int_{\R^n}|f(x)|^2 (\rho+|x|^2)^sd\gamma(x)\right)^{1/2}\left(\int_{\R^n}\frac{f(x)^2}{(\rho+|x|^2)^s}d\gamma(x)\right)^{1/2}$$ which along with Theorem \ref{h-u_s} gives 
  \begin{cor}
  	For any $f\in L^2(\gamma)$ with $U_sf\in L^2(\gamma)$, we have 
  	$$ \left(\int_{\R^n}|f(x)|^2 (\rho+|x|^2)^sd\gamma(x)\right) \langle U_sf,f\rangle_{L^2(\gamma)} \ge (2\rho)^s \frac{\Gamma\left(\frac{n/2+1+s}{2}\right) }{\Gamma\left(\frac{n/2+1-s}{2}\right) } \left(\int_{\R^n} |f(x)|^2 d\gamma(x)\right)^2.$$
  \end{cor}
  We must mention that   one can use the $L^2$ boundedness of $U_sL^{-s}$    along with the inequality for $U_s$ to derive inequality for $L^s$. Indeed, the operator norm of $\mathfrak{R}_s:=U_sL^{-s}$ is given by 
  $$\|\mathfrak{R}_s\|_{op}=\sup_{k\geq 0} 2^s(2k+n)^{-s} \frac{\Gamma\left( \frac{k+n/2+1+s}{2}\right)}{\Gamma\left( \frac{k+n/2+1-s}{2}\right)}$$ which can be estimated as above to get $\|\mathfrak{R}_s\|_{op}\leq 2^{-s} $. This together with the fact that $\|\mathfrak{R}_s\|_{op}\langle L^sf,f\rangle_{L^2(\gamma)}\ge \langle U_sf,f\rangle_{L^2(\gamma)}$ yields 
  \begin{thm}[Hardy's inequality for $L^s$]
  	   	Let $0<s<1$. Assume that $f\in L^2(\gamma)$ such that $L^sf\in L^2(\gamma)$.  Then for any $\rho>0$ we have 
  	  \[ \langle L^sf,f\rangle_{L^2(\gamma)} \ge (4\rho)^s \frac{\Gamma\left(\frac{n/2+1+s}{2}\right) }{\Gamma\left(\frac{n/2+1-s}{2}\right) } \int_{\R^n}\frac{f(x)^2}{(\rho+|x|^2)^s}d\gamma(x). \]
  \end{thm}   
    
    The main ingredient in proving the above mentioned trace Hardy and Hardy's inequality for fractional powers of $L$ is the  solution of the extension problem for $L$: 
    	\begin{equation}
    	\label{ep1}
    	\Big( -L + \partial_\rho^2 + \frac{1-2s}{\rho} \partial_\rho - \frac{1}{4}\rho^2 \Big) u(x,\rho) = 0,\,\, u(x,0) = f(x) .
    	\end{equation} 
    	As can be seen later that a solution of the above partial differential equation will play a very crucial role for our purpose.  The second theme of this article is the study of general solutions of the extension problem for $L$ under consideration. In fact, we prove a characterization  of the solution when the initial data is a tempered distribution. In order to state the result we need to introduce some more notations which will be  briefly described  here. More details can be found in Section 3.
 We introduce the following two operators.  For any   distribution $f$ for which $ M_\gamma f $ is tempered, we define
  $$S^1_{\rho}f(x) :=\frac{(\frac{1}{2}\rho^2)^{\frac{s-1}{2}}}{\Gamma(s)}\sum_{k=0}^{\infty}\Gamma\left(\frac12(2k+n+s+1)\right)W_{-(k+n/2),s/2}( \rho^2/2)\,\,Q_kf(x) $$  and for  any  function $g$ for which  $Q_kg$ has enough decay as a function of $ k $ we define
 $$	S^2_{\rho}g(x)=\left(\frac{1}{2}\rho^2\right)^{\frac{s-1}{2}}\sum_{k=0}^{\infty}M_{-(k+n/2),s/2}( \rho^2/2)\,\,Q_kg(x)$$ where $W_{-(k+n/2),s/2} $ and $M_{-(k+n/2),s/2}$ are Whittaker functions. 
 
 In view of the asymptotic properties of the Whittaker functions stated in Lemma \ref{Walarge} it follows that the series  defining $ S_\rho^1 f $ converges for any tempered distribution $ M_\gamma f.$  Moreover, if we take $ g $ from $ H^2_{\gamma, \rho}(\R^n) $ which is the image of $ L^2(\R^n, \gamma) $ under the semigroup $ e^{- \rho \sqrt{L}} $ then the series defining $ S_\rho^2 g $ also converges and defines a smooth function. With these notations we prove the following characterization:       

    \begin{thm}
    	 Let $ f $ be a distribution  such that $ M_\gamma f$ is tempered. Then any function  $u(x,\rho)$   for which $ M_\gamma u(x,\rho) $ is tempered in $ x, $ is a solution of the extension problem \ref{ep1}  with initial condition $ f $ if and only if $u(x,\rho)=S^1_{\rho}f(x)+S^2_{\rho}g(x)$ for some $g \in  \cap_{t>0}H^2_{\gamma,t}(\R^n).$ 
    \end{thm}

  We also prove another characterization of the solution of the extension problem in terms of its holomorphic extendability. In order to state this we need to introduce some more notations. 
  For any $ t, \delta > 0 $ we consider the following positive weight function
  $$w_t^{\delta}(x,y) =\frac{1}{\Gamma(\delta)}  \int_{\R^n} e^{-2ux}  \left( 1- \frac{|u|^2+|y|^2}{t^2}\right)_{+}^{\delta-1} \, e^{-(|u|^2+|y|^2)} \, du .$$  For any $\rho>0$ we let  $H^2(\Omega_\rho,w^{2s}_{\rho})$  stand for the weighted Bergman space consisting of holomorphic functions on the tube domain $\Omega_\rho:=\{z=x+iy\in\mathbb{C}^n:|y|<\rho\}$ belonging to $L^2(\Omega_\rho, w^{2s}_{\rho}).$  Also for $m\in \R$, let $W^m_{H}(\R^n)$ stand for the Sobolev space associated to the Hermite operator $H$. This is a  Hilbert space in which the norm is given by 
$$\|f\|^2_{W^m_{H}}:=\sum_{k=0}^{\infty}(2k+n)^{2m}\|P_kf\|_2^2.$$
\begin{thm}
	\label{hLLL}
		A solution of the extension problem \ref{ep1} is of the form $u(x,\rho)=S^1_{\rho}f(x)$ for some distribution $f$ such that   $M_{\gamma}f\in W^{m_n}_{H}(\R^n)$ where $2m_n=-(2n+1)/4$ if and only if for every $\rho>0$, $M_{\gamma} u(.,\rho)$ extends holomorphically to $\Omega_{\rho/2}$ and   satisfies the uniform estimate 
	\begin{equation*}
	\int_{\Omega_{ \rho/2}}  | M_{\gamma}u(z,\rho)|^2  w^{2s}_{ \rho/2}(z) dz \leq C \rho^{n-1/2}   
	\end{equation*} 
	for all $ 0 < \rho \leq 1.$
\end{thm}

   We conclude the introduction by describing the plan of the paper. In section 2, we study an extension problem for the Ornstein-Uhlenbeck operator. We provide two representations of  solutions and their equivalence. In section 3, we prove several characterisations of the solution of the extension problem under consideration.    Using the results obtained in Section 2, we prove trace Hardy and Hardy's inequality in section 4. Then in section 5, we prove an isometry property of the solution to the extension problem. Finally we end our discussion by proving a Hardy-Littlewood-Sobolev type inequality for the fractional powers of Hermite operator in section 6.
    
  \section{The extension problem for the Ornstein-Uhlenbeck operator\\
  	and fractional powers}
  \subsection{The extension problem} Our strategy to prove Hardy's inequality for $ L_s $ is via trace Hardy inequality which in turn requires the study of the following  extension problem for the operator $ L$:
  	\begin{equation}
  	\label{ep} 
  	\Big( -L + \partial_\rho^2 + \frac{1-2s}{\rho} \partial_\rho - \frac{1}{4}\rho^2 \Big) u(x,\rho) = 0,\,\, u(x,0) = f(x) .
  	\end{equation} 
  	If $ u $ is a  solution of the above problem, it follows that $ v(x,\rho) = M_\gamma u(x,\rho) $ solves the problem
  	\begin{equation}
  	\label{epH}
  	\Big( -H + \partial_\rho^2 + \frac{1-2s}{\rho} \partial_\rho - \frac{1}{4}\rho^2 \Big) v(x,\rho) = 0,\,\, v(x,0) = M_\gamma f(x) .
  	\end{equation}
  	A solution of the above problem can be obtained in terms of the solution of an extension problem for the sublaplacian on the Heisenberg group.  
  
   Let $ \mathcal{L} $ be the sublaplacian on the Heisenberg group $ \mathbb{H}^n .$ Then a solution of the following extension problem   
  	$$  \Big( -\mathcal{L} + \partial_\rho^2 + \frac{1-2s}{\rho} \partial_\rho +\frac{1}{4}\rho^2 \partial_t^2 \Big) w(z,t,\rho) = 0,\,\, w(z,t,0) = f(z,t) $$    is given by (see \cite{RT2})
  	$ w(z,t,\rho) = \rho^{2s} f \ast \Phi_{s,\rho}(z,t)$, where $\Phi_{s,\rho}$ is an explicit function given by $$\Phi_{s,\rho}(z,a) = \frac{2^{-(n+1+s)} }{ \pi^{n+1}\Gamma(s)}\Gamma\left(\frac{n+1+s}{2}\right)^2\left((\frac14\rho^2+\frac{1}{4}|z|^2)^2+a^2\right)^{-\frac{n+1+s}{2}}. $$ 
	  If we let $ \pi $ stand for the Schrodinger representation of $ \mathbb{H}^n $ on $ L^2(\R^n). $ Then 
  	we have the following result. 
  
   \begin{thm} \label{thm21}  For any $ f \in L^2(\gamma) $ the function  $ v(x,\rho) $ defined by the equation
  		$$ v(x,\rho)  =   \rho^{2s} \int_{\mathbb{H}^n}  \Phi_{s,\rho}(g) \pi(g)^\ast M_\gamma f(x)  dg $$ solves the extension problem for the Hermite operator with initial condition $ M_\gamma f$.
  		Consequently, $ u(x,\rho) = e^{\frac{1}{2}|x|^2} v(x,\rho) $ solves the extension problem for $ L.$
  	\end{thm} 
  	\begin{proof} For any $ X $ from the Heisenberg Lie algebra $ \mathfrak{h}^n $ viewed as a left invariant vector field on $ \Ha$ we can easily check that
  		$$ \pi(X) \int_{\Ha} \varphi(g) \pi_(g)^\ast f(x) dg = -\int_{\Ha}  X\varphi(g) \pi(g)^\ast f(x) dg.$$ This leads to 
  		$$  H \int_{\Ha} \varphi(g) \pi_\lambda(g)^\ast f(x) dg = \int_{\Ha} \mathcal{L}\varphi(g) \pi_\lambda(g)^\ast f(x) dg$$
  		and consequently, as $$  \rho^{2s} \mathcal{L} \Phi_{s,\rho}(g) = 
  		\Big(  \partial_\rho^2 + \frac{1-2s}{\rho} \partial_\rho + \frac{1}{4}\rho^2 \partial_t^2 \Big) \rho^{2s} \Phi_{s,\rho}(g) = 0 $$
  		we obtain
  		$$ \Big(  -H + \partial_\rho^2 + \frac{1-2s}{\rho} \partial_\rho - \frac{1}{4}\rho^2  \Big) \Big( \rho^{2s} \int_{\mathbb{H}^n}  \Phi_{s,\rho}(g) \pi(g)^\ast f(x)  dg \Big)  = 0. $$
  		In order to check that $ v(x,\rho) $ satisfies the initial condition, we make the change of variables $ (z,t) \rightarrow (\rho z,\rho^2 t) $ so that
  		$$ v(x,\rho)  =    \int_{\mathbb{H}^n}  \Phi_{s,1}(z,t) \pi(\rho z,\rho^2 t)^\ast M_\gamma f(x)  dz\,dt .$$
  		Since $ \pi(\rho z,\rho^2 t) M_\gamma f$ converges to $ M_\gamma f $ in $ L^2(\R^n) $ we obtain $ v(x,\rho) \rightarrow M_\gamma f $ as $ \rho \rightarrow 0 $ in view of $ \int_{\Ha} \Phi_{s,1}(g) dg =1.$
  		This completes the proof of the theorem.
  \end{proof} 
  
    There is yet another convenient way of representing the solution of the extension problem for $ L.$  If we  let  $k_{t,s}(\rho)=(\sinh t)^{-s-1}e^{-\frac{1}{4}(\coth t)\rho^2},$  then it is known that this function satisfies the equation
  	$$  \partial_tk_{t,s}(\rho)=\left(\partial_\rho^2+\frac{1+2s}{\rho}\partial_{\rho}-\frac{1}{4}\rho^2\right)k_{t,s}(\rho).$$ 
   \begin{thm}
  		\label{epsolth}
  		For $f\in L^p(\gamma)$, $1\le p\le \infty$, a solution of the extension problem for $ L $  is given by
  		\begin{equation}
  		\label{epsol}
  		u(x,\rho)=\frac{4^{-s}}{\Gamma(s)}\rho^{2s}\int_{0}^{\infty}  k_{t,s}(\rho) e^{-tL}f(x ) dt .
  		\end{equation}
  		Moreover, as $\rho\rightarrow0$, the solution $u(.,\rho)$ converges to $f$ in $L^p(\gamma)$ for any $1\le p<\infty.$  
  \end{thm} 
  \begin{proof} That $ u $ solves the extension problem follows easily from the fact that $ e^{-tL}f(x )$ solves the heat equation associated to $L$, i.e., $ -L e^{-tL}f(x) =\partial_t e^{-tL}f(x)$ and the definition of $ k_{t,s}(\rho).$ Indeed we have 
  	$$ -Lu(x,\rho)=\frac{4^{-s}}{\Gamma(s)}\rho^{2s}\int_{0}^{\infty}  k_{t,s}(\rho) \partial_tv(x,t) dt $$  which after an integration by parts in the $t$ variable yields
  	$$ Lu(x,\rho)=\frac{4^{-s}}{\Gamma(s)}\rho^{2s}\int_{0}^{\infty}  \partial_t k_{t,s}(\rho) v(x,t) dt.$$  Since $ k_{t,s}(\rho)$ is  the heat kernel associated to the operator $ \big(\partial_\rho^2+\frac{1+2s}{\rho}\partial_{\rho}-\frac{1}{4}\rho^2 \big)$ we have 
  	$$Lu(x,\rho)=\frac{4^{-s}}{\Gamma(s)}\rho^{2s}\left(\partial_\rho^2+\frac{1+2s}{\rho}\partial_{\rho}-\frac{1}{4}\rho^2\right)\int_{0}^{\infty}   k_{t,s}(\rho) e^{-tL}f(x) dt.$$ Finally, an easy calculation shows that for any function $ v(\rho),$	 one has
  	$$ \left(\partial_\rho^2+\frac{1-2s}{\rho}\partial_{\rho}-\frac{1}{4}\rho^2\right)(\rho^{2s} v(\rho))=\rho^{2s}\left(\partial_\rho^2+\frac{1+2s}{\rho}\partial_{\rho}-\frac{1}{4}\rho^2\right) v(\rho)$$ and hence it follows that $ u (x,\rho) $ solves the extension problem.
  	
  	Now to prove the $L^p(\gamma)$ convergence of the solution to the initial condition, we make use of the fact that $ e^{-tL} $ is a contraction semigroup on every $ L^p(\gamma) $ and $ e^{-tL}f $ converges to $ f $ as $ t \rightarrow 0 $ in $ L^p(\gamma).$ We first make a change of variables $t\rightarrow \rho^2t$ to get 
  	\[u(x,\rho)=\frac{4^{-s}}{\Gamma(s)}\rho^{2s+2}\int_{0}^{\infty}  k_{\rho^2t,s}(\rho) e^{-\rho^2tL}f(x ) dt.\]
  	Note that       
  	\begin{align}
  	\label{limit}
  	\rho^{2s+2}k_{\rho^2t,s}(\rho)&= \rho^{2s+2}(\sinh \rho^2t)^{-s-1}e^{-\frac{1}{4}(\coth \rho^2t)\rho^2} \nonumber\\
  	&= t^{-s-1}\left(\frac{\rho^2t}{\sinh \rho^2t}\right)^{s+1}e^{-\frac{1}{4t}(\coth \rho^2t)\rho^2t}\nonumber\\
  	& \rightarrow t^{-s-1}e^{-\frac{1}{4t}} ~~\text{as} ~~\rho  \rightarrow 0. 
  	\end{align} 
  	Here we have used the facts that as $y\rightarrow 0$, $\frac{\sinh y}{y}\rightarrow 1$ and $y\coth y\rightarrow 1.$
  	Also we see that $t^{-s-1}e^{-\frac{1}{4t}}\in L^1(0,\infty)$ and an easy calculation yields 
  	\begin{align*}
  	\int_{0}^{\infty }t^{-s-1}e^{-\frac{1}{4t}}dt=4^s\Gamma(s).
  	\end{align*}
  	Now using this result we can write for any $x\in\mathbb{R}^n$
  	\begin{align*}
  	u(x,\rho)-f(x)&=\frac{4^{-s}}{\Gamma(s)}\rho^{2s+2}\int_{0}^{\infty}  k_{\rho^2t,s}(\rho) e^{-\rho^2tL}f(x)~ dt-\frac{4^{-s}}{\Gamma(s)}\int_{0}^{\infty }t^{-s-1}e^{-\frac{1}{4t}}dtf(x)\\
  	&= \frac{4^{-s}}{\Gamma(s)}\int_{0}^{\infty}\left(\rho^{2s+2}k_{\rho^2t,s}(\rho)-t^{-s-1}e^{-\frac{1}{4t}}\right) e^{-\rho^2tL}f(x)dt\\& \hspace*{1cm}+\frac{4^{-s}}{\Gamma(s)}\int_{0}^{\infty}t^{-s-1}e^{-\frac{1}{4t}} \left( e^{-\rho^2tL}f(x)-f(x)\right)dt
  	\end{align*}
  	Therefore, using Minkowski's integral inequality and the fact that $\|e^{-\rho^2tL}f\|_{L^p(\gamma)}\le  \|f\|_{L^p(\gamma)}$ we have
  	\begin{align}
  	\label{lpsol}
  	\|u(.,\rho)-f\|_{L^p(\gamma)}&\le \frac{4^{-s}}{\Gamma(s)}\int_{0}^{\infty}\left| \rho^{2s+2}k_{\rho^2t,s}(\rho)-t^{-s-1}e^{-\frac{1}{4t}}\right| \|f\|_{L^p(\gamma)}dt \\ &\hspace*{1cm}+\frac{4^{-s}}{\Gamma(s)}\int_{0}^{\infty}t^{-s-1}e^{-\frac{1}{4t}} \|e^{-\rho^2tL}f-f\|_{L^p(\gamma)}dt
  	\end{align}
  	Note that using the asymptotics of sine and cotangent hyperbolic functions we have  $$|\rho^{2s+2}k_{\rho^2t,s}(\rho)-t^{-s-1}e^{-\frac{1}{4t}}|\leq C \rho^{2s+2}e^{-\rho^2t(s+1)}+t^{-s-1}e^{-\frac{1}{4t}}:=h_{\rho}(t), t>M$$ It is not hard to see that for every $\rho>0$, $h_{\rho}\in L^1 $ and $\lim_{\rho\rightarrow0}\int_{M}^{\infty}h_{\rho}(t)dt=\int_{M}^{\infty}h(t)dt  $ where as $\rho\rightarrow0$, $h_{\rho}(t)\rightarrow t^{-s-1}e^{-\frac{1}{4t}}=:h(t)$ pointwise. Hence by generalised dominated convergence theorem we have   \[\int_{M}^{\infty}\left| \rho^{2s+2}k_{\rho^2t,s}(\rho)-t^{-s-1}e^{-\frac{1}{4t}}\right| \|f\|_{L^p(\gamma)}dt\rightarrow 0~ as~ \rho\rightarrow0.\]  Now see that similarly as in \ref{limit}, one can show that  as $t\rightarrow0$ the function $h_{\rho}(t)$ goes to a finite limit. So there is no singularity of $h_{\rho}$ at $0$. Hence it is easy to see that $$\int_{0}^{M}\left| \rho^{2s+2}k_{\rho^2t,s}(\rho)-t^{-s-1}e^{-\frac{1}{4t}}\right| \|f\|_{L^p(\gamma)}dt$$ goes to zero as $\rho\rightarrow0.$
  	Hence it follows that the first integral in the RHS of \ref{lpsol}  goes to zero. Also the integrand of second integral is bounded above by an integrable function of $t$. Indeed,
  	\[t^{-s-1}e^{-\frac{1}{4t}} \|e^{-\rho^2tL}f-f\|_{L^p(\gamma)}\le 2t^{-s-1}e^{-\frac{1}{4t}} \|f\|_{L^p(\gamma)}.\] Hence by DCT the second integral goes to zero as $\rho\rightarrow0.$
  	Therefore, we have 
  	\begin{align*}
  	u(.,\rho)=\frac{4^{-s}}{\Gamma(s)}\rho^{2s+2}\int_{0}^{\infty}  k_{\rho^2t,s}(\rho) e^{-\rho^2tL}f~ dt \rightarrow f ~~ \text{in}~~ L^p(\gamma) ~~\text{as}~~ \rho\rightarrow 0.  
  	\end{align*}
  	This completes the proof.	 
  \end{proof}	
  
    We have thus given two representations for solutions of the extension problem but they are the same. This is not obvious and  needs a proof.  It is convenient to work with the functions
  	$$ \varphi_{s,\delta}(z,a) = ((\delta+\frac{1}{4}|z|^2)^2+a^2)^{-\frac{n+1+s}{2}}$$ 
  	in terms of which we can express $ \Phi_{s,\rho} $ as follows:  with $ \delta = \frac{1}{4}\rho^2,$
  	\begin{equation}  \Phi_{s,\rho}(z,a) = \frac{2^{-(n+1+s)} }{ \pi^{n+1}\Gamma(s)}  \Gamma\left(\frac{n+1+s}{2}\right)^2 \varphi_{s,\delta}(z,a).
  	\end{equation}
  	For a function $ \varphi(z,t) $ on $ \Ha $ we let $ \varphi^\lambda(z) $ to stand for the inverse Fourier transform of $ \varphi $ in the $ t $ variable. Thus
  	$$  \varphi_{s,\delta}^\lambda(z) = \int_{-\infty}^\infty \varphi_{s,\delta}(z,t) e^{i\lambda t} dt.$$
  	This is  a radial function on $ \Ca $ and hence has an expansion in terms of the Laguerre functions
  	\begin{equation}
  	\label{lagu}
  	  \varphi_k^\lambda(z) = L_k^{n-1}(\frac{1}{2}|\lambda| |z|^2) e^{-\frac{1}{4}|\lambda| |z|^2}.
  	\end{equation}
      We let $ c_{k,\delta}^\lambda(s) $ to be the coefficients defined by
  	\begin{equation}  \varphi_{s,\delta}^\lambda(z) = (2\pi)^{-n} |\lambda|^n \sum_{k=0}^\infty  c_{k,\delta}^\lambda(s) \varphi_k^\lambda(z) .\end{equation}
  	These coefficients are given in terms of the auxiliary function $ L(a,b,c) $ defined for $ a, b \in \R_+ $ and $ c \in \R$ as follows:
  	\begin{equation}
  	L(a,b,c) = \int_0^\infty e^{-a(2x+1)}x^{b-1}(1+x)^{-c} dx.
  	\end{equation}
  	The following proposition expresses $ c_{k,\delta}^\lambda(s)$ in terms of $ L,$   see \cite{CH}.
  	\begin{prop} [Cowling-Haagerup] For any $ \delta > 0 $ and  $ 0 < s <  \frac{n+1}{2} $ we have
  		$$  c_{k,\delta}^\lambda(s) = \frac{(2\pi)^{n+1}|\lambda|^s}{ \Gamma(\frac{n+1+s}{2})^2} L\left(\delta |\lambda|, \frac{2k+n+1+s}{2}, \frac{2k+n+1-s}{2}\right).$$
  	\end{prop}

    Using  this proposition  we can compute the explicit formula for the group Fourier transform of $ \Phi_{s,\rho}(g)$ on $ \Ha.$ Let $ P_k(\lambda) $ stand for the projections associated to $ H(\lambda) = -\Delta+\lambda^2 |x|^2$. Then making use of the fact that 
  	$$  \int_{\Ca} \varphi_k^\lambda(z)  \pi_\lambda(z,0) dz  = (2\pi)^{-n} |\lambda|^{-n} P_k(\lambda) $$ we obtain the  following formula: with $ \delta = \frac{1}{4}\rho^2,$ as before,
  	$$  \int_{\Ha} \Phi_{s,\rho}(g) \pi_\lambda(g)^\ast  dg =  \frac{2^{-(n+1+s)} }{ \pi^{n+1}\Gamma(s)}  \Gamma(\frac{n+1+s}{2})^2  \sum_{k=0}^\infty  c_{k,\delta}^\lambda(s) P_k(\lambda). $$
  	As the projections associated to $ L $ are given by $ Q_k = M_\gamma^{-1} P_k M_\gamma $ we see that the solution defined in  Theorem 2.1 is given by 
  	$$ u(x,\rho) =  \frac{2^{-(n+1+s)} }{ \pi^{n+1}\Gamma(s)}  \Gamma(\frac{n+1+s}{2})^2    \rho^{2s} \sum_{k=0}^\infty  c_{k,\delta}^1(s) Q_k f(x) .$$ 
  	Therefore, in order to prove our claim, we only need check if 
  	$$\frac{4^{-s}}{\Gamma(s)}\rho^{2s}\int_{0}^{\infty}  k_{t,s}(\rho) e^{-tL}f(x ) dt = \frac{2^{-(n+1+s)} }{ \pi^{n+1}\Gamma(s)}  \Gamma(\frac{n+1+s}{2})^2    \rho^{2s} \sum_{k=0}^\infty  c_{k,\delta}^1(s) Q_k f(x) $$
  	where $ \delta = \frac{1}{4} \rho^2.$ Equivalently, we need to check if 
  	$$ \int_0^\infty  k_{t,s}(\rho) e^{-t(2k+n)} dt = L\left(\frac14 \rho^2, \frac{2k+n+1+s}{2}, \frac{2k+n+1-s}{2}\right).$$ 
  
    In order to compute the above integral, we make the change of variables  $\coth t =2z+1$ and note that $-(\sinh^2 t)^{-1} \, dt=2dz $ and $\sinh t=(2z \, (2z+2))^{-\frac{1}{2}}.$ We get
  	\begin{eqnarray*}
  		&& \int_{0}^{\infty}  (\sinh t)^{-s-1} \, e^{-\frac{1}{4}(\coth t) \rho^2} \, e^{-t(2k+n)} \, dt\\
  		&&=2\int_{0}^{\infty}  (2z(2z+2))^{\frac{1}{2}(s-1)} \, e^{-\frac{1}{4}(2z+1) \rho^2} \, \left(\frac{2z+2}{2z}\right)^{-\frac{1}{2}(2k+n)} \, dz\\
  		&&= 2\int_{0}^{\infty}  \, e^{-\frac{1}{4}(2z+1) \rho^2} \, (2z)^{\frac{1}{2}[(s-1)+(2k+n)]}  (2z+2)^{-\frac{1}{2}[(1-s)+(2k+n)]} \, dz\\
  		&&= 2^s\int_{0}^{\infty}  \, e^{-\frac{1}{4}(2z+1) \rho^2} \, (z)^{\frac{1}{2}[(s+1)+(2k+n)]-1}  (z+1)^{-\frac{1}{2}[(1-s)+(2k+n)]} \, dz\\
  		&&=2^s \, L\left(\frac{1}{4} \, \rho^2, \frac{2k+n+1+s}{2},\frac{2k+n+1-s}{2}\right).
  	\end{eqnarray*}
  	This proves our claim that Theorems \ref{thm21} and \ref{epsolth} define the same solution of the extension problem. 
  
  The above proof also shows that the function $ u(x,\rho) $ defined  by the integral
  $$  u(x,\rho)=\frac{4^{-s}}{\Gamma(s)}\rho^{2s}\int_{0}^{\infty}  k_{t,s}(\rho) e^{-tU}f(x ) dt $$ 
  using $ U $ in place of $ L $ solves the extension problem for $ U$ 
  and  the following expansion for the solution $ u $ is valid.
  
  \begin{prop} \label{prop2.4} For $ 0 < s < \frac{n+1}{2} $ and $ f \in L^2(\gamma) $ the solution of the extension problem associated to $ U $ is given by
  	\begin{equation}
  	u(.,\rho)=	 \frac{2^{-s}}{\Gamma(s)} \, \rho^{2s} \, \sum_{k=0}^\infty  L\left(\frac{1}{4} \, \rho^2,\frac{k+n/2+1+s}{2},\frac{k+n/2+1-s}{2}\right) \, Q_k f.
  	\end{equation}
  \end{prop}
  
    We let $ T_{s,\rho}$ stand for the solution operator which takes $f $ into the solution $ u(x,\rho) $ of the extension problem. Thus
   $$ T_{s,\rho}f(x) = \frac{4^{-s}}{\Gamma(s)}\rho^{2s}\int_{0}^{\infty}  k_{t,s}(\rho) e^{-tL}f(x ) dt$$
   which is also given by the expansion  in the above proposition. In what follows we make use of the transformation property
   \begin{equation} \label{3.7}
   \frac{(2\lambda)^a}{\Gamma(a)} L(\lambda, a, b)=\frac{(2\lambda)^b}{\Gamma(b)} L(\lambda, b, a)
   \end{equation}
   satisfied by the $ L $ function for all admissible values of $ (a,b,c)$, see Cowling-Haagerup \cite{CH}.   
   \subsection{Fractional powers of the operators $ L$ and $ U$}   In what follows let $ A $ stand for either $ L $ or $ U.$ Note that the associated eigenvalues $ \lambda_k $ are given by $ (2k+n) $ and $ (k+n/2) $ respectively. The above representation of the solution of the extension problem allows us to  define $ A_s $ as the Neumann boundary data associated to the extension problem. More precisely we have the following result. 
   
   \begin{thm}  Assume that $0<s<1$. Let $f\in L^2\cap L^p(\gamma)$, $1\le p<\infty$ be  such that $A_sf\in L^p(\gamma)$.  Then the  solution of the extension problem $ u(x,\rho) = T_{s,\rho} f(x) $ satisfies
   	\[\displaystyle\lim_{\rho\rightarrow0} \rho^{1-2s}\partial_\rho u(x,\rho)=-2^{1-2s} \, \frac{\Gamma(1-s)}{\Gamma(s)} \, A_sf 
   	\]  where the convergence is understood in the $L^p(\gamma)$ sense. 
   \end{thm}
   
   \begin{proof} The expansion of $ T_{s,\rho}f $ given in Proposition \ref{prop2.4} and the transformation property \eqref{3.7} of the $ L $ function allows us to verify  the following identity:
   	\begin{equation}
   	\label{Ls1}
   	\rho^{2s}T_{-s,\rho} (A_sf)(x)= \frac{4^{s} \Gamma(s)}{\Gamma(-s)} \, T_{ s,\rho} f(x) 
   	\end{equation}
   	which when expanded reads as 	  
   	$$
   	\frac{4^{s} \Gamma(s)}{\Gamma(-s)} \, u(x,\rho)  = \, \frac{4^s}{\Gamma(-s)}  \, \int_0^{\infty} (\sinh t)^{s-1} \, e^{-\frac{1}{4}(\coth t) \, \rho^2} \, e^{-tA}A_sf(x) \, dt.
   	$$
   	Differentiating with respect to $\rho$ and multiplying both sides by $-\rho^{1-2s}$, we get
   	
   	\begin{eqnarray*}
   		-\rho^{1-2s} \,	\partial_{\rho}u(x,\rho)=  \frac{1}{2\Gamma(s)} \rho^{2(1-s)} \, \int_0^{\infty} (\sinh t)^{s-1} \, (\coth t) \, e^{-\frac{1}{4}(\coth t) \, \rho^2} \, e^{-tA} A_sf(x) \, dt.
   	\end{eqnarray*}
   	Now, we make the change of variables $t\to t \rho^2,$ to get
   	\begin{eqnarray*}
   		&&	-\rho^{1-2s} \,	\partial_{\rho}u(x,\rho)\\
   		&&=  \, \frac{1}{2\Gamma(-s)} \rho^{4-2s} \, \int_0^{\infty} (\sinh (t \rho^2))^{s-1} \, (\coth (t \rho^2)) \, e^{-\frac{1}{4}(\coth (t \rho^2)) \, \rho^2} \, e^{-t \rho^2L} L_sf(x)  \, dt\\
   		&&= \, \frac{1}{2\Gamma(-s)}  \, \int_0^{\infty} t^{s-2} \, \big(\frac{\sinh (t \rho^2)}{t\rho^2}\big)^{s-1} \, \coth (t \rho^2)(t\rho^2) \, e^{-\frac{1}{4} \coth (t \rho^2) \, \rho^2} \, e^{-t \rho^2 A} A_sf(x)\, dt.
   	\end{eqnarray*}
   	Under  the extra assumption that  $A_sf \in L^p(\gamma),~1\leq p<\infty$  we know that  $\displaystyle\lim_{\rho\rightarrow0} e^{-\rho^2 t A}A_sf=A_sf,$ in $ L^p(\gamma).$  So as $\rho \to 0$, 
   	we can argue as in the proof of the Theorem \ref{epsolth} to obtain 
   	$$\displaystyle\lim_{\rho\rightarrow0} \left(-\rho^{1-2s} \, \partial_{\rho} u(x,\rho)\right)=  \, \frac{1}{2\Gamma(s)} \, A_sf\left(\int_0^{\infty} t^{s-2} \, e^{-\frac{1}{4t}} dt\right).  $$
   	Computing the last integral and simplifying we obtain
   	$$ \displaystyle\lim_{\rho\rightarrow0} \left(\rho^{1-2s} \, \partial_{\rho} u(x,\rho)\right) =-2^{(1-2s)} \, \frac{\Gamma(1-s)}{\Gamma(s)} \, A_sf.  $$
   \end{proof} 	
  	\section{Characterisations of solutions of the extension problem}
  	In this section we prove several characterisations of solutions of the extension problem for $L$.  Recall that the extension problem for $L$ reads as 
  	$$  \Big( -L + \partial_\rho^2 + \frac{1-2s}{\rho} \partial_\rho - \frac{1}{4}\rho^2 \Big) u(x,\rho) = 0,\,\, u(x,0) = f(x) .$$ Now given $\alpha\in\mathbb{N}^n$ and $\rho>0$ we define the Fourier-Hermite coefficients associated to the expansion in terms of the normalised Hermite polynomials $H_{\alpha}$ as 
  	\[\tilde{u}(\alpha,\rho) :=\int_{\R^n}u(x,\rho)H_{\alpha}(x)d\gamma(x).\] 
  	Now letting $v_{\alpha}(\rho):=\tilde{u}(\alpha,2\sqrt{\rho})$, we see that 
  	\[\Big( -(2|\alpha|+n) + \rho\partial_\rho^2 + (1-s) \partial_\rho -  \rho \Big) v_{\alpha}(\rho)= 0,\,\,    v_{\alpha}(0)= (f,H_{\alpha})_{L^2(\gamma)}.\]
  	
  	Again if we write $v_{\alpha}(\rho)=e^{-\rho}g_{\alpha}(2\rho)$, then it can be easily checked that the above equation changes to 
  	\[rg_{\alpha}''(r)+(1-s-r)g_{\alpha}'(r)-\frac{2|\alpha|+n+1-s}{2}g_{\alpha}(r)=0\]where  $r=2\rho.$ Now we let $g_{\alpha}(r)=r^sh_{\alpha}(r)$ which leads to 
  	\begin{equation}
  	\label{kform}
  	rh_{\alpha}''(r)+(1+s-r)h_{\alpha}'(r)-\frac{2|\alpha|+n+1+s}{2}h_{\alpha}(r)=0 . 
  	\end{equation}
  	Note that this is in the form of Kummer's equation $xh^{"}(x)+(b-x)h^{'}-ah(x)=0$.  The  solutions of the Kummer's equation are given by the functions  $M(a,b,x)$ and $V(a,b,x)$ which are known as the confluent hypergeometric functions. The function $M$ is given by $M(a,b,x)=\sum_{m=0}^{\infty}\frac{(a)_m}{(b)_mm!}x^m$  is analytic and   
  		\[V(a,b,x)=\frac{\pi}{\sin \pi b}\left(\frac{M(a,b,x)}{\Gamma(1+a-b)\Gamma(b)}-x^{1-b}\frac{M(1+a-b,2-b,x)}{\Gamma(a)\Gamma(2-b)}\right),\ \ x>0.\] Also $V$ has the  integral representation given by 
  		\[V(a,b,x)=\frac{1}{\Gamma(a)}\int_{0}^{\infty}e^{-tx}t^{a-1}(1+t)^{b-a-1}dt,~~x>0.\] For more details see for instance \cite[Chapter 13]{AS} and also \cite[Lemma 5.2]{FGMT}.   
		
		Finally writing  $\mu=\frac{s}{2}$ and $\kappa =  |\alpha|+n/2 $  another substitution $w_{\alpha}(r)=e^{-\frac{1}{2}r} r^{\frac{1}{2}+\mu} h_{\alpha}(r), $ transform the equation \ref{kform} to 
		\begin{equation}
		\label{weq}
		w_{\alpha}''(r) +\left( -\frac{1}{4}-\frac{\kappa }{r}+\frac{1/4-\mu^2}{r^2}\right)w_{\alpha}(r)=0
		\end{equation} which is in the form of a Whittaker equation. This warrants the following lemma which describes the properties of solutions of Whittaker equation.
		\begin{lem}[\cite{OM}]
		Let $\kappa\in\R$ and $-2\mu\neq \mathbb{N}$. The two linearly independent solutions of the ordinary differential equation $$w''(x) +\left( -\frac{1}{4}+\frac{\kappa}{x}+\frac{1/4-\mu^2}{x^2}\right)w(x)=0$$
are given by the functions $M_{\kappa,\mu}(x)$	 and $W_{\kappa,\mu}(x)$	 where $$M_{\kappa,\mu}(x)=e^{-x/2}x^{1/2+\mu}\sum_{p=0}^{\infty}\frac{\frac12+\mu-\kappa}{(1+2\mu)_pp!}x^p$$ and when $2\mu$ is not an integer 
$$W_{\kappa,\mu}(x)=\frac{\Gamma(-2\mu)}{\Gamma(1/2-\mu-\kappa)}M_{\kappa,\mu}(x)+\frac{\Gamma(+2\mu)}{\Gamma(1/2+\mu-\kappa)}M_{\kappa,-\mu}(x).$$ Moreover we have the following asymptotic properties these Whittaker functions.

For large $x$ \begin{align} &M_{\kappa,\mu}(x)\sim \frac{\Gamma(1+2\mu)}{\Gamma(1/2+\mu-\kappa)}e^{x/2}x^{-\kappa},~ \mu-\kappa\neq -1/2,-3/2,...\text{and}\\
&W_{\kappa,\mu}(x)\sim e^{-x/2}x^{\kappa}.\end{align} Also as $x\rightarrow 0$ we have 
\begin{align}\label{wasmall}
	&M_{\kappa,\mu}(x)= x^{\mu+1/2}(1+O(x)),~2\nu\neq -1,-2,-3,...~\text{and}\\&W_{\kappa,\mu}(x)=\frac{\Gamma(2\mu)}{\Gamma(1/2+\mu-\kappa)}x^{1/2-\mu}+\frac{\Gamma(-2\mu)}{\Gamma(1/2-\mu-\kappa)}x^{1/2+\mu}+O(x^{3/2-\mu}) , ~0< \mu<1/2, 
	\end{align}
\end{lem}

 In view of the above lemma  generic solutions of \ref{weq} are given by 
 $$w_{\alpha}(r)=C_{1}(|\alpha|) M_{-(|\alpha|+n/2),s/2}(r)+C_2(|\alpha|)W_{-(|\alpha|+n/2),s/2}(r).$$  But we know  $v_{\alpha}(\rho)=e^{-\rho}g_{\alpha}(2\rho)=e^{-\rho}(2\rho)^sh_{\alpha}(\rho)= e^{-\rho}(2\rho)^se^{\rho/2}\rho^{-1/2-\mu}w_{\alpha}(\rho)$ and by definition $v_{\alpha}(\rho) =\tilde{u}(\alpha,2\sqrt{\rho}).$  Hence we have
  	\begin{align}
  	\label{FHsol}
  	\tilde{u}(\alpha,\rho)=\left(\frac{1}{2}\rho^2\right)^{\frac{s-1}{2}} \left(C_1(|\alpha|)W_{-(|\alpha|+n/2),s/2}( \rho^2/2)+C_2(|\alpha|)M_{-(|\alpha|+n/2),s/2}( \rho^2/2)\right).
  	\end{align}
 The initial condition on the solution along with behaviour of the Whittaker  functions stated in the previous lemma allows us to conclude that 
$$ C_1(|\alpha|) = \frac{1}{\Gamma(s)} \Gamma\left(\frac12(2k+n+s+1)\right) (f,H_{\alpha})_{L^2(\gamma)}.$$ 
Thus the solution of the extension problem can be written as a sum of two functions, namely
$$ (\frac{1}{2}\rho^2)^{\frac{s-1}{2}} \frac{1}{\Gamma(s)}\sum_{k=0}^{\infty}\Gamma\left(\frac12(2k+n+s+1)\right) W_{-(k+n/2),s/2}( \rho^2/2)Q_kf$$
and  another function given by the series 
$$  (\frac{1}{2}\rho^2)^{\frac{s-1}{2}}  \sum_{\alpha \in \N^n}  C_2(|\alpha|)M_{-(|\alpha|+n/2),s/2}( \rho^2/2)  H_\alpha(x) .$$  
The above series converges under some decay conditions on the coefficients $ C_2(|\alpha|)$ as we will see soon. We make use of these considerations in the proof of Theorem \ref{ch}  below.
  	
  	To proceed further with our description of solutions of the extension problem, we need  the following asymptotic properties of Whittaker functions appearing in the above expressions for large values of the parameter $ k.$   
  	\begin{lem} 
  		\label{Walarge}
  		For  any $\rho \in (0,\infty),$ we have the following asymptotic properties, as $ k $ tends to infinity, we have 
  		\begin{equation}\label{m}
  		\left(\frac12\rho^2\right)^{\frac{s-1}{2}}M_{-(k+n/2),s/2}(\rho^2/2)\sim (\rho)^{s-1/2}(\sqrt{2k+n})^{-s-1/2} e^{2(2k+n) \zeta\left(\frac{\rho^2}{4(2k+n)}\right)^{\frac{1}{2}} }
  		\end{equation} 
  		\begin{equation}\label{w}
  		\left(\frac12\rho^2\right)^{\frac{s-1}{2}}W_{-(k+n/2),s/2}(\rho^2/2)\sim \frac{(\rho\sqrt{2k+n})^{s-1/2}}{\Gamma\left(\frac12(2k+n+1+s)\right)}  e^{-2(2k+n) \zeta\left(\frac{\rho^2}{4(2k+n)}\right)^{\frac{1}{2}} },
  		\end{equation}
  			where $2\sqrt{\zeta(x)}=\sqrt{x+x^2}+ \ln(\sqrt{x}+\sqrt{x+1}),~x>0.$
  	\end{lem}
  	\begin{proof}
  		For large value of $\kappa$ and for any $x\in (0,\infty)$ the following asymptotic properties can be found in    \cite[13.21.6,13.21.7]{OM}    
  		\begin{equation}
  		\label{Mor}
  		M_{-\kappa,\mu}(4 \kappa x)= \frac{2\Gamma(2\mu+1)}{\kappa^{\mu-\frac{1}{2}}} \left(\frac{x\zeta(x )}{1+x} \right)^{\frac{1}{4}} \, I_{2\mu}\left(4\kappa \zeta\left(x\right)^{\frac{1}{2}}\right)\left( 1+O(\kappa^{-1}) \right) 
  		\end{equation}
  		\begin{equation}
  		\label{Wor}
  		W_{-\kappa,\mu}(4 \kappa x)=\frac{\sqrt(8/\pi) e^{\kappa}}{\kappa^{\kappa-\frac{1}{2}}} \left(\frac{x\zeta(x )}{1+x} \right)^{\frac{1}{4}} \, K_{2\mu}\left(4\kappa \zeta\left(x\right)^{\frac{1}{2}}\right)\left( 1+O(\kappa^{-1}) \right)
  		\end{equation}
  	  where $I_{ 2\mu}$ is the modified Bessel function of first kind and $K_{2\mu}$ denotes the Macdonald function of order $2\mu$.  
  		Taking $x=\frac12\rho^2$ , $\kappa=k+n/2$ and $\mu=s/2$, for large value of $k$, from \ref{Mor}  we have   
  		\begin{align*}
  		&M_{-(k+n/2),s/2}(x)\\&= \frac{2\Gamma(2\mu+1)}{(k+n/2)^{\mu-\frac{1}{2}}} \left(\frac{x\zeta(x/2(2k+n))}{2(2k+n)+x} \right)^{\frac{1}{4}} \, I_{2\mu}\left(2(2k+n) \zeta\left(\frac{x}{ 2(2k+n)}\right)^{\frac{1}{2}}\right)\left( 1+O(k^{-1}) \right)
  		\end{align*}
  		Recall that   the modified Bessel function of first kind which has the following asymptotic property:
  		\begin{equation}
  		\label{mb}
  		I_{2\mu}(x)\sim \frac{1}{\sqrt{2\pi x}} e^x ~\text{when $x$ is real and } ~x\rightarrow \infty.
  		\end{equation}
  		But it is easy to see that as $k\rightarrow \infty$, $2(2k+n) \zeta\left(\frac{x}{ 2(2k+n)}\right)^{\frac{1}{2}}$ goes to infinity which by the above asymptotic property yields
  		\begin{equation}
  		\label{Ias}
  		I_{2\mu}\left(2(2k+n) \zeta\left(\frac{x}{ 2(2k+n)}\right)^{\frac{1}{2}}\right)\sim \left(2(2k+n) \zeta\left(\frac{x}{ 2(2k+n)}\right)^{\frac{1}{2}}\right)^{-1/2}e^{(2(2k+n) \zeta\left(\frac{x}{ 2(2k+n)}\right)^{\frac{1}{2}}}
  		\end{equation} 
  		valid for large values of $k$. 
  		It can be easily checked that for any $x>0$ and large $k$ 
  	\begin{equation}
  		\label{estx}
  		\left(\frac14\right)^{ \frac14}\left(\frac{x}{2k+n}\right)^{\frac14}\leq \left(\frac{x}{2(2k+n)+x}\right)^{\frac14}\leq \left(\frac34\right)^{\frac14} \left(\frac{x}{2k+n}\right)^{\frac14}
  		\end{equation}  
  		This, along with \ref{Ias}  proves the result for the function $M_{-(k+n/2),s/2}$.   
  		
  		Similarly we can obtain the  asymptotic property for the other function. Indeed, for large $k$ from \ref{Wor}  we have 
  		\begin{align*}
  		&W_{-(k+n/2),s/2}(x)=\frac{\sqrt(8/\pi) e^{(k+n/2)}}{(k+n/2)^{(k+n/2)-\frac{1}{2}}} \left(\frac{x\zeta(x/2(2k+n))}{2(2k+n)+x} \right)^{\frac{1}{4}} \, \\
  		&K_{2\mu}\left(2(2k+n) \zeta\left(\frac{x}{ 2(2k+n)}\right)^{\frac{1}{2}}\right)\left( 1+O(k^{-1}) \right).
  		\end{align*}
  		Now the Macdonald's function $K_{2\mu}(z)$  has the following asymptotic property:
  		\begin{equation}
  		\label{mac}
  		K_{2\mu}(x)\sim \sqrt{\frac{\pi}{2x}}e^{-x},~\text{when $x$ is real and } ~x\rightarrow \infty.
  		\end{equation} Again for same reason as above as $k\rightarrow\infty$, using \ref{mac} we have
  		 \begin{equation}
  		 \label{Kas}
  		 K_{2\mu}\left(2(2k+n) \zeta\left(\frac{x}{ 2(2k+n)}\right)^{\frac{1}{2}}\right)\sim \left(2(2k+n) \zeta\left(\frac{x}{ 2(2k+n)}\right)^{\frac{1}{2}}\right)^{-1/2}e^{-2(2k+n) \zeta\left(\frac{x}{ 2(2k+n)}\right)^{\frac{1}{2}}}.
  		 \end{equation}
  	Using the Stirling's formula   $\Gamma(x)=\sqrt{2\pi}x^{x-1/2}e^{-x}e^{\theta(x)/12x},~0<\theta(x)<1$ true for $x>0$ (See Ahlfors \cite{A}), we have
  		$$\frac{\Gamma\left(\frac12(2k+n+1+s)\right)e^{(k+n/2)}}{(k+n/2)^{(k+n/2)-\frac{1}{2}}}= \frac{\Gamma\left(\frac12(2k+n+1+s)\right)}{e^{-\theta(k+n/2)/6(2k+n)}\Gamma\left(\frac12(2k+n) \right)}\sim \left(\frac12(2k+n)\right)^{(1+s)/2},$$ 
		$\text{as}~k\rightarrow \infty$. 
  		This observation along with \ref{estx} and the asymptotic property \ref{Kas}  yields \[\Gamma\left(\frac12(2k+n+1+s)\right) W_{-(k+n/2),s/2}(x)\sim  (2k+n)^{s/2-1/4}x^{1/4}e^{-2(2k+n) \zeta\left(\frac{x}{ 2(2k+n)}\right)^{\frac{1}{2}}}\] completing the   proof of the lemma. 
  	\end{proof}
  \begin{rem}
  	It can be easily checked that for large $\kappa$ the following inequality is valid for any $x>0$:
  	\begin{equation}
  	\label{zeta}
  	\frac12 \sqrt{x\kappa}\leq\kappa \sqrt{\zeta(x/\kappa)}\leq \frac32 \sqrt{x\kappa}
  	\end{equation} 
  	 which can be used to further simplify the exponential part in the above estimates.  
  \end{rem}
  	
  	 The analysis preceding Lemma \ref{Walarge} motivates us to define the following two operators. Given a  distribution $f$ such that $M_{\gamma}f$ is a tempered distribution   we define    
	\begin{equation}\label{S1}
	S^1_{\rho}f= \frac{(\frac{1}{2}\rho^2)^{\frac{s-1}{2}}}{\Gamma(s)}\sum_{k=0}^{\infty}\Gamma\left(\frac12(2k+n+s+1)\right)W_{-(k+n/2),s/2}( \rho^2/2)Q_kf. 
	\end{equation} Recall  that $h$ is a tempered distribution on $\R^n$ if and only if the Hermite coefficients  satisfy the estimate $|(h,\Phi_{ \alpha})|\le C(2|\alpha|+n)^m$ for some integer $m$.  So $M_{\gamma}f$ being a tempered distribution, its Hermite coefficients have at most polynomial growth and consequently  $Q_kf$ has polynomial growth in $k$. So  because of the exponential decay in \ref{w}, the above series defining $S^1_{\rho}f$ converges uniformly. Consequently, in view of \ref{FHsol}, $S^1_{\rho}f$ defines a solution of the extension problem. 
	
	Also considering the other solution of the Whittaker equation we define the operator $S_{\rho}^2$  for nice functions $g$ by  
	\begin{equation}\label{S2}
	S^2_{\rho}g=\left(\frac{1}{2}\rho^2\right)^{\frac{s-1}{2}}\sum_{k=0}^{\infty}M_{-(k+n/2),s/2}( \rho^2/2)Q_kg.
	\end{equation}
 It is not hard to see that as the Whittaker function $M_{-(k+n/2),s/2}(\rho^2/2)$ has exponential growth as $k\rightarrow\infty$, $Q_kg$ must have enough decay for the series in \ref{S2} to converge. This encourages us to determine condition on the function $g$ so that the projections $Q_kg$ have enough decay. Now as can be seen in the above lemma, the function $M_{-(k+n/2),s/2}(\rho^2/2)$ is growing like $e^{c\rho\sqrt{2k+n}}$ for large value of $k$ which leads us to consider the image of $L^2(\gamma)$ under the semigroup $e^{-tL^{\frac12}}$ which we denote by $H^2_{\gamma,t}(\R^n)$. Clearly if $g  \in  \cap_{t>0} H^2_{\gamma,t}(\R^n)$, the series in \ref{S2} converges and defines a smooth function.  But in view of the connection between $L$ and the Hermite operator $H$, we note that  a function $g$ is in $H^2_{\gamma,t}(\R^n)$ if and only if $ge^{-|.|^2/2}$ is in the image of $L^2(\R^n)$ under the poisson semigroup $e^{-tH^{\frac12}}$. Let us write $H^2_t(\R^n):=e^{-tH^{\frac12}}(L^2(\R^n)).$   
  	We are ready to prove the following characterization for the solution of the extension problem.  
  	\begin{thm}\label{ch}
  		 Let $ f $ be a distribution  such that $ M_\gamma f$ is tempered. Then any function  $u(x,\rho)$   for which $ M_\gamma u(x,\rho) $ is tempered in $ x, $ is a solution of the extension problem \ref{ep1}  with initial condition $ f $ if and only if $u(x,\rho)=S^1_{\rho}f(x)+S^2_{\rho}g(x)$ for some $g \in  \cap_{t>0}H^2_{\gamma,t}(\R^n).$  
  	\end{thm}  
  	
  	\begin{proof}
  		First suppose  $u(x,\rho)=S^1_{\rho}f(x)+S^2_{\rho}g(x)$ for some $g$ such that $g \in\cap_{t>0}H^2_{\gamma, t}(\R^n).$ Consequently we have for every $t>0$, $\|Q_kg\|^2_{L^2(\gamma)}\leq C e^{-2t(2k+n)^{1/2}}$ for large $k$. So   the expression \ref{S2} defining $S^2_{\rho}g$ is well-defined and solves the extension problem.  
  		
  		Now since $M_{\gamma}f$ is tempered distribution , as mentioned above, the Fourier-Hermite coefficients  associated to Hermite polynomials of $f$ satisfies   
  		\[|\tilde{f}(\alpha)|=|(f,H_{\alpha })_{L^2(\gamma)}|\leq C (2|\alpha|+n)^{m} \text{ for some integer m.  }\] But in view of the fact that $\sum_{|\alpha|=k}1 = \frac{(k+n-1)!}{k!(n-1)!} \leq C (2k+n)^{n-1}$ we have $\|Q_kf\|^2_{L^2(\gamma)}\leq C (2k+n)^{2m+n-1}.$   Now  the   asymptotic property \ref{m} in Lemma \ref{Walarge} along with the estimate \ref{zeta} gives 
  		$$\left(\frac12\rho^2\right)^{\frac{s-1}{2}}\Gamma\left(\frac12(2k+n+1+s)\right)W_{-(k+n/2),s/2}(\frac{1}{2}\rho^2)\leq (\rho\sqrt{2k+n})^{s-1/2}e^{-\frac12\rho\sqrt{2k+n}} $$
  		 which allow us to conclude that    
  		$$\sum_{k=0}^{\infty}\left(\Gamma\left(\frac12(2k+n+1+s)\right)W_{-(k+n/2),s/2}(\frac{1}{2}\rho^2)\right)^2(2k+n)^{2m+n-1}<\infty.$$ Consequently, 
  		 $S^1_{\rho}f$ make sense and hence solves the extension problem.   Now we observe that an easy calculation yields 	
  \begin{align}
  \label{LV}
&\frac{(\frac{1}{2}\rho^2)^{\frac{s-1}{2}}}{\Gamma(s)}\Gamma\left(\frac12(2k+n+1+s)\right)W_{-(k+n/2),s/2}(\frac{1}{2}\rho^2)\\
&=\frac{2^{-s}}{\Gamma(s)}\rho^{2s}L\left(\frac{1}{4} \, \rho^2, \frac{2|\alpha|+n+1+s}{2},\frac{2|\alpha|+n+1-s}{2}\right),\notag 
\end{align} 
 which together with the expression \ref{S1} yields that $S^1_{\rho}f$ is in the form \ref{epsol} and as discussed in the previous subsection, this converges to $f$ as $\rho\rightarrow0$.  Also note that from the asymptotic property in Lemma \ref{wasmall}, we have $(\frac{1}{2}\rho^2)^{\frac{s-1}{2}}M_{-(k+n/2),s/2}(\frac{1}{2}\rho^2)$ approaches to zero as $\rho\rightarrow0$. So $S^2_{\rho}g\rightarrow0$ as $\rho\rightarrow0$. Therefore, $u=S^1_{\rho}f+S^2_{\rho}g$ solves the extension problem with initial condition $f$.

  		Conversely, suppose $u(x,\rho)$ is a solution of the extension problem \ref{ep} with initial condition $f$ whose Fourier-Hermite coefficients associated to the Hermite polynomials have tempered growth.  Then as discussed in the beginning of this subsection we have
  		$$\tilde{u}(\alpha,\rho)= (\frac{1}{2}\rho^2)^{\frac{s-1}{2}} \left(C_1(|\alpha|)W_{-(k+n/2),s/2}(\frac{1}{2}\rho^2)+C_2(|\alpha|)M_{-(k+n/2),s/2}(\frac{1}{2}\rho^2)\right).$$
		 Now using $\tilde{u}(\alpha,0)=(f,H_{\alpha})$ and behaviour of $(\frac{1}{2}\rho^2)^{\frac{s-1}{2}}  W_{-(k+n/2),s/2}(\frac{1}{2}\rho^2)$ near $\rho=0$ (see Lemma \ref{wasmall}) we have 
		$$C_1(|\alpha|)=\frac{\Gamma\left(\frac12(2|\alpha|+n+1+s)\right)}{\Gamma(s)}(f,H_{\alpha})_{L^2(\gamma)}.$$
		 Also since $M_{\gamma}u(x,\rho)$ is tempered, $\tilde{u}(\alpha,\rho)$ has atmost polynomial growth in $|\alpha|$. But the estimate \ref{zeta} along with the asymptotic property \ref{m} yield 
		 $$\left(\frac12\rho^2\right)^{\frac{s-1}{2}}M_{-(k+n/2),s/2}(\rho^2/2)\leq C (\rho)^{s-1/2}(\sqrt{2k+n})^{-s-1/2} e^{\frac32\rho\sqrt{2k+n}}$$ for large $k$. 
		  Hence we must have $C_2(|\alpha|)$ decaying as $e^{-\frac32\rho(2|\alpha|+n)^{1/2}}$ for every $\rho>0.$    
		So let us take  $g=\sum_{  \alpha \in\mathbb{N}^n }C_{2}( |\alpha|)H_{\alpha }$.  Then the function $g$ satisfies  $\|Q_kg\|^2_{L^2{(\gamma})}\leq Ce^{-\frac32\rho(2k+n)^{1/2}}$ for every $\rho>0$.  This ensures that $g\in H^2_{\gamma,3\rho/2}(\R^n)$ for every $\rho>0 $ which completes the proof.
  	\end{proof} 
  \begin{rem}
  	 For any $ \rho > 0 $ the space $ H^2_{\rho}(\R^n) $ has an interesting characterisation. 
  	   It is well-known that any $ g $ from this space has a holomorphic extension   to the tube domain  $\Omega_{\rho} = \{z=x+iy\in\C^n: |y|< \rho \} $  in $ \C^n $  which belongs to $ L^2(\Omega_\rho, w_\rho) $ for an explicit positive weight function $w_\rho$ given by 
  	   $$w_{\rho}(z)=(\rho^2-|y|^2)^{n/2} \frac{J_{n/2-1}(2i(\rho^2-|x|^2)^{1/2}|x|)}{(2i(\rho^2-|x|^2)|x|)^{n/2-1}},~z=x+iy\in \C^n $$ 	where $J_{n/2-1}$ denotes the Bessel function of order $(n/2-1).$ 
  	     We denote this weighted Bergman space by $ H^2_{\rho}(\C^n) .$
  	     In \cite{T1}, Thangavelu proved that for any holomorphic function $F$ on $\Omega_{\rho}$ 
  	     \begin{equation}
  	     \label{Psc}
  	      \int_{\Omega_{ \rho }}|F(z)|^2w_{\rho}(z)dz= c_n\sum_{k=0}^{\infty}\|P_kf\|^2 \frac{k!(n-1)!}{(k+n-1)!}L^{n-1}_k(-2\rho^2)e^{\rho^2} 
  	     \end{equation}
  	     where $f$ is the restriction of $F$ to $\R^n. $ In view of this identity we
  	      see that   $ g \in  H^2_{\rho}(\R^n) $ if and only if  the function $M_\gamma g$ extends holomorphically to $\Omega_{\rho}$ and belongs to $  H^2_\rho(\C^n) .$ We refer the reader to \cite{T1} for more details in this regard. From this observation we infer that the condition $g \in  \cap_{t>0}H^2_{\gamma,t}(\R^n) $ in the above theorem can be replaced by the requirement  $M_{\gamma}g$ extends holomorphically and belongs to $\cap_{t>0}H^2_{ t}(\C^n).$
  \end{rem}
  	We also have the following characterization of the solution $u(x,\rho)$ when $M_{\gamma}u(x,\rho)$ has tempered growth in both  the variables. 
  	\begin{thm}
  		Suppose $u(x,\rho)$ is a solution of the extension problem \ref{ep} with $M_{\gamma}u$ is tempered   (in both variables). Then $u=S^1_{\rho}f$ for some $f\in L^p(\gamma)$ if and only if $\sup_{\rho>0}\|u(.,\rho)\|_{L^p(\gamma)}\leq C.$
  	\end{thm}
  	\begin{proof}
  		Suppose $f\in L^p(\gamma)$ and let $u=S^1_{\rho}f$. Then as mentioned earlier
  		$$u(x,\rho)=\frac{4^{-s}}{\Gamma(s)}\rho^{2s}\int_{0}^{\infty}  k_{t,s}(\rho) e^{-tL}f(x ) dt.$$ Now since $e^{-tL}$ is a contraction semigroup on $L^p(\gamma)$ we have 
  		$$\|u(.,\rho)\|_{L^p(\gamma)}\leq \|f \|_{L^p(\gamma)}\frac{4^{-s}}{\Gamma(s)}\rho^{2s}\int_{0}^{\infty}  k_{t,s}(\rho)  dt$$
  		Proceeding similarly as before one can easily see that 
  		$$\int_{0}^{\infty}  k_{t,s}(\rho)  dt=2^{s+1}L\left(\frac{1}{4}\rho^2,\frac{1+s}{2}, \frac{1-s}{2}\right).$$ So we have  
  		$$\|u(.,\rho)\|_{L^p(\gamma)}\leq C_s\|f \|_{L^p(\gamma)}\rho^{2s}L\left(\frac{1}{4}\rho^2,\frac{1+s}{2}, \frac{1-s}{2}\right). $$ Now we make use of an estimate for L function, which can be found in \cite[Page-18]{RT3}, to get 
  		$$\|u(.,\rho)\|_{L^p(\gamma)}\leq C_s\|f \|_{L^p(\gamma)}\rho^{2s}\Gamma(s)\left(\rho^2/2\right)^{-s}e^{-\rho^2/4}=2\|f\|_{L^p(\gamma)}e^{-\rho^2/4}$$ which gives the required boundedness.

  		Conversely,  let $\sup_{\rho>0}\|u(.,\rho)\|_{L^p(\gamma)}\leq C.$ This condition  allows us to extract a subsequence $\rho_j$ along which $u(.,\rho)$ converges weakly to a function $f\in L^p(\gamma)$. Now letting $\rho  $ go to zero along $\rho_j$ from equation \ref{FHsol} we have 
  		\begin{align*} &\tilde{u}(\alpha,\rho) \left(\frac{1}{2}\rho^2\right)^{\frac{s-1}{2}} =\\
		& \left(\frac{\Gamma\left(\frac12(2|\alpha|+n+1+s)\right)}{\Gamma(s)}(f,H_{\alpha})_{L^2(\gamma)}W_{-(k+n/2),s/2}(\rho^2/2)+C_2(|\alpha|)M_{-(k+n/2),s/2}(\rho^2/2)\right).
		\end{align*}
  		Now as $\rho \rightarrow\infty$ we have 
  		$$  M_{-(k+n/2),s/2}(\rho^2/2)\sim \frac{\Gamma(1+2\mu)}{\Gamma(1/2+\mu+(k+n/2))}e^{\rho^2/4}\rho^{(2k+n)}.$$
		 But it is given that $\tilde{u}(\alpha,\rho)$ has polynomial growth in $\rho$ variable. So we must have $C_2(\alpha)=0$ and hence we are done. 
  	\end{proof}
  	
   Now we turn our attention to the holomorphic extendability of solutions of the extension problem under consideration. To motivate what we plan to do, we first recall a result about holomorphic extendability of solutions of the following extension problem for the Laplacian on $\R^n$:
$$\left( \Delta+\partial_{\rho}^2+\frac{1- s}{\rho}\partial_{\rho} \right)u(x,\rho)=0,~\ \ u(x,0)=f(x),~~\ \ x\in\R^n,\rho>0.$$   
 After the remarkable work of 
 Caffarelli and Silvestre \cite{CS}, this problem has been extensively studied in the literature. See e.g., the work of Stinga-Torrea \cite{ST}.  It is known that  for $f\in L^2(\R^n)$  the function  $u(x,\rho)=\rho^sf\ast\varphi_{s,\rho}(x)$ where $\varphi_{s,\rho}$ is the generalised poisson kernel given by $$\varphi_{s,\rho}(x)=\pi^{-n/2}\frac{\Gamma(\frac{n+s}{2})}{|\Gamma(s)|}(\rho^2+|x|^2)^{-\frac{n+s}{2}},~x\in \R^n$$ is a solution of the extension problem. Recently in \cite{RT3} authors proved that a necessary and sufficient condition for the solution of the above problem to be of the form $u(x,\rho)=\rho^sf\ast\varphi_{s,\rho}(x)$ for some $f\in L^2(\R^n)$  is that for every $\rho>0$, $u(.,\rho)$ extends holomorphically to the tube domain $\Omega_{ \rho } $ in $\C^n$, belongs to a weighted Bergman space $B_{s}(\Omega_{ \rho })$ and satisfies the uniform estimate $\|u(.,\rho)\|_{B_s}\leq C $ for all $\rho>0 $ where the norm $\|.\|_{B_s}$ is given by $$\|F\|_{B_s}^2:=\rho^{-n}\int_{\Omega_{ \rho }}|F(x+iy)|^2\left(1-\frac{|y|^2 }{\rho^2}\right)_{+}^{s-1}dxdy.$$ Our aim in the rest of this section is to prove an analogous result for the extension problem we considered for the Ornstein-Uhlenbeck operator $L$.   In order to do so,  we require the following Gutzmer's formula for the Hermite expansions. In order to state the same,  we need to introduce some more notations.

   Let $Sp(n,\R)$ denote the symplectic group consisting of $2n\times 2n$
real matrices which preserves the symplectic form $[(x,u),(y,v)]=(u.y-v.x)$ on $\R^{2n}$ with determinant 1. Suppose $K:=Sp(n,\R)\cap O(2n,\R)$ where $O(2n,\R)$ stands for the orthogonal group.  For a complex matrix $\sigma=a+ib$, it is known that $\sigma $ is unitary if and only if the matrix $\sigma_A:=\begin{pmatrix}
a & -b\\ b &a
\end{pmatrix}$ belongs to the group $K$ which yields a one to one correspondence between $K$ and the unitary group $U(n)$.  A proof of which can be found in Folland \cite{Fol}.  We let $\sigma.(x,u)$ stand for the action of $\sigma_A$ on $(x,u)$ which clearly has a natural extension to $\C^n\times \C^n$.  Also given $(x,u)\in \R^n\times \R^n$, let $\pi(x,u)$ be the unitary operator acting on $L^2(\R^n)$ defined by 
\begin{align*}
\pi(x,u)\phi(\xi)= e^{i(x.\xi+\frac12x.u)}\phi(\xi+u),~ \xi\in \R^n.
\end{align*}
Clearly for $(z,w)\in \C^n\times C^n$, $\pi(z,w)\phi(\xi)$ makes perfect sense as long as  $\phi$ is holomorphic.  Also note that Laguerre functions of type $(n-1)$, defined earlier in \ref{lagu}, can be considered as a function on $\R^n\times \R^n$ which can be holomorphically extended to $\C^n\times \C^n$ as follows:
$$\varphi_k(z,w):=L_k^{n-1}(\frac{1}{2}(z^2+w^2) ) e^{-\frac{1}{4}(z^2+w^2)},~z,w\in \C^n.$$
 We have the following very useful identity proved in Thangavelu\cite{T12}. 
\begin{thm}[Gutzmer's formula]
	For a holomorphic function $f$ on $\mathbb{C}^n,$ we have the following formula: 
	\[\int_{\R^n} \int_K |\pi (\sigma \cdot (z,w)) \, f(\xi)|^2 d \sigma d\xi=e^{(u\cdot y-v\cdot x)} \, \sum_{k=0}^{\infty} \frac{k!(n-1)!}{(k+n-1)!} \, \varphi_k(2iy,2iv) \, \|P_kf\|_2^2 \] 
	where  $z=x+iy,~~w=u+iv \in \mathbb{C}^n.$
\end{thm}
   We use this to prove the following result:
 	\begin{prop}
 		\label{identity}
 	 Let $\delta>0$. 	For a holomorphic function $F$ on $\Omega_t$ we have the following identity
 		\[\int_{\R^n} \int_{|y|<t} |F(x+iy)|^2 w_t^{\delta}(x,y) \, dx dy=C_n \sum_{k=0}^{\infty} \|P_kf\|_2^2    \frac{\Gamma(k+1) \Gamma(n+\delta)}{\Gamma(k+n+\delta)}   \, L_k^{n+\delta-1}(-2t^2) t^{2n},\]
 		where $f$ denotes the restriction of $F$ to $\R^n$ and the weight $w_t^{\delta} >  0$ is given by
 		\[w_t^{\delta}(x,y) =\frac{1}{\Gamma(\delta)}  \int_{\R^n} e^{-2ux}  \left( 1- \frac{|u|^2+|y|^2}{t^2}\right)_{+}^{\delta-1} \, e^{-(|u|^2+|y|^2)} \, du .\]
 		
 		\begin{proof}
 			Let $F$ be holomorphic  in the  tube domain $\Omega_t=\{z=x+iy:|y|<t\}$  of $\mathbb{C}^n$.  Now since the Lebesgue measure is rotation invariant,  $\left( 1- \frac{|u|^2+|y|^2}{t^2}\right)_{+}^{\delta-1} \, e^{-(|u|^2+|y|^2)} \, dy du$ is a rotation invariant measure. So, using Gutzmer's formula, we have
 			\begin{align}
 			\label{formula}
 			&  \int_{\R^{2n}} \left( \int_{\R^n}|\pi(iy,iv) F(\xi)|^2 d\xi\right)  \left( 1- \frac{|u|^2+|y|^2}{t^2}\right)_{+}^{\delta-1} \, e^{-(|u|^2+|y|^2)} \, dy du\nonumber\\
 			&=c_n   \sum_{k=0}^{\infty} \|P_kf\|_2^2  \frac{k!(n-1)!}{(k+n-1)!}  \, \int_{\R^{2n}} \varphi_k(2iy,2iu) \left( 1- \frac{|u|^2+|y|^2}{t^2}\right)_{+}^{\delta-1} \, e^{-(|u|^2+|y|^2)} \, dy du.
 			\end{align}
			Integrating in polar coordinates,  the integral on the right hand side becomes 
 			\begin{align*}
 			&\int_{\R^{2n}} \varphi_k(2iy,2iu) \left( 1- \frac{|u|^2+|y|^2}{t^2}\right)_{+}^{\delta-1} \, e^{-(|u|^2+|y|^2)} \, dy du
 			 \\
			 &\,\,\,\,=\omega_{2n} \int_0^{\infty} L_k^{n-1}(-2r^2) \, \left( 1- \frac{r^2}{t^2}\right)_{+}^{\delta-1}  \, r^{2n-1} dr.
 			\end{align*}
 			Now using a change of variable $r \rightarrow rt$ followed by another change of variable $r \rightarrow \sqrt{r}$ in the  integral in the RHS of the above equation, we have   
 			\begin{align*}
 			&  \int_0^{\infty} L_k^{n-1}(-2r^2) \, \left( 1- \frac{r^2}{t^2}\right)_{+}^{\delta-1}  \, r^{2n-1} dr= \frac{1}{2}t^{2n} \int_0^{1} L_k^{n-1}(r(-2 t^2)) \, \left( 1- r\right)^{\delta-1}  \, r^{n-1} dr .
 			\end{align*}
 		By making  use of the following identity  (see \cite{SG})
 			\[L_k^{\alpha}(t)=\frac{\Gamma(k+\alpha+1) }{\Gamma(\alpha-\beta)\Gamma(k+\beta+1)} \int_0^1(1-r)^{\alpha-\beta-1} \, r^{\beta} \, L_k^{\beta}(rt) \, dr\]  
			the above yields
 			 \begin{align}
 			   &\frac{1}{\Gamma(\delta)}\int_{\R^{2n}} \varphi_k(2iy,2iu) \left( 1- \frac{|u|^2+|y|^2}{t^2}\right)_{+}^{\delta-1} \, e^{-(|u|^2+|y|^2)} \, dy du\nonumber\\&=\frac{1}{2}t^{2n} \omega_{2n}  \frac{  \Gamma(k+n) }{\Gamma(k+n+\delta)} \,  L_k^{n+\delta-1}(-2 t^2).
 			 \end{align}
 			 Now simplify the LHS of the equation \ref{formula}:
 			\begin{align*}
 			& \frac{1}{\Gamma(\delta)} \int_{\R^{2n}} \left( \int_{\R^n}|\pi(iy,iv) F(\xi)|^2 d\xi\right)  \left( 1- \frac{|u|^2+|y|^2}{t^2}\right)_{+}^{\delta-1} \, e^{-(|u|^2+|y|^2)} \, dy du\\
 			&= \frac{1}{\Gamma(\delta)} \int_{\R^{2n}} \left( \int_{\R^n}|e^{i(iy\cdot \xi+\frac{1}{2}iy\cdot iu)}F(\xi+iu)|^2 d\xi\right)  \left( 1- \frac{|u|^2+|y|^2}{t^2}\right)_{+}^{\delta-1} \, e^{-(|u|^2+|y|^2)} \, dy du\\
 			&= \frac{1}{\Gamma(\delta)} \int_{\R^{2n}} \left( \int_{\R^n}|e^{-2y\cdot \xi} \, F(\xi+iu)|^2 d\xi\right)  \left( 1- \frac{|u|^2+|y|^2}{t^2}\right)_{+}^{\delta-1} \, e^{-(|u|^2+|y|^2)} \, dy du\\
 			&= \int_{\R^{2n}}  |F(\xi+iu)|^2  \left( \frac{1}{\Gamma(\delta)} \int_{\R^n} e^{-2y\cdot \xi}  \left( 1- \frac{|u|^2+|y|^2}{t^2}\right)_{+}^{\delta-1} \, e^{-(|u|^2+|y|^2)} \, dy\right)d \xi \, du\\
 			&=\int_{\R^{2n}}  |F(\xi+iu)|^2 w_t^{\delta}(u,\xi) \, d \xi \, du.
 			\end{align*}
 			Now, we see that  when $|u|\geq t,$  $\left( 1- \frac{|u|^2+|y|^2}{t^2}\right)_{+}^{\delta-1} =0 $ for all $y\in \R^n.$ Thus,
 			\[\int_{\R^{2n}}  |F(\xi+iu)|^2 w_t^{\delta}(u,\xi) \, d \xi \, du=\int_{\R^{n}} \int_{|u|<t}  |F(\xi+iu)|^2 w_t^{\delta}(u,\xi) \, d \xi \, du.\]
 			Finally, we have
 			\[\int_{\R^{n}} \int_{|u|<t}  |F(\xi+iu)|^2 w_t^{\delta}(u,\xi) \, d \xi \, du =c_n \sum_{k=0}^{\infty} \|P_kf\|_2^2  \frac{\Gamma(k+1) \Gamma(n+\delta)}{\Gamma(k+n+\delta)} \, L_k^{n+\delta-1}(-2t^2) t^{2n}.\]
 			This completes the proof of the theorem.
 		\end{proof}
 		
 	\end{prop}
 
   For $s>0$ we consider the following positive weight function $  \tilde{w}_{\rho}(k)$ on $\mathbb{N}$ given by the sequence
  $$
  \left(\frac12\rho^2\right)^{s-1}\left(\Gamma\left(\frac12(2k+n+1+s)\right)W_{-(k+n/2),s/2}(\frac{1}{2}\rho^2)\right)^2\frac{\Gamma(k+1) \Gamma(n+2s)}{\Gamma(k+n+2s)} L^{n+2s-1}_k(-\frac12\rho^2) . $$
 We define $W^s_{\rho}(\R^n)$ to be the space of all tempered distribution $f$ for which 
  $$\|f\|_{s,\rho}^2:=\sum_{k=0}^{\infty}\tilde{w}_{\rho}(k)\|P_kf\|_2^2<\infty.$$ 
  
  \begin{rem}
  	For $r<0$ the following asymptotic property of Laguerre function is well known (see\cite[Theorem 8.22.3]{SG}) and is valid for large $k$ and for $r\leq -c,~c>0$ 
  	\begin{equation}
  	\label{lag}
  	L^{\alpha}_k(r)=\frac{1}{2\sqrt{\pi}}e^{r/2}(-r)^{-\alpha/2-1/4}k^{\alpha/2-1/4}e^{2(-kr)^{\frac12}}(1+O(k^{-1/2})).
  	\end{equation}
  	Also the asymptotic property \ref{w} and \ref{zeta} together gives 
  	$$\left(\frac12\rho^2\right)^{s-1}\left(\Gamma\left(\frac12(2k+n+1+s)\right)W_{-(k+n/2),s/2}(\frac{1}{2}\rho^2)\right)^2\leq c_1  (\rho\sqrt{2k+n})^{2s-1}e^{-\rho\sqrt{2k+n}} $$ and from \ref{lag} we have  
	$$L^{n+2s-1}_k(-\rho^2/2)\leq c e^{\rho^2/4} \rho^{-n-2s+\frac12}(2k+n)^{\frac{n+2s-1}{2}-\frac14}e^{\rho\sqrt{2k+n}}. $$ Now using the fact   that $\frac{\Gamma(k+1) \Gamma(n+2s)}{\Gamma(k+n+2s)} \sim (2k+n)^{-(n+2s-1)}$  we  have $$\tilde{w}_{\rho}(k)\leq c_1e^{\rho^2/4}  (\rho^2(2k+n))^{-\frac{2n+1}{4}}.$$	
  	On the other hand, using \ref{w} and \ref{lag}, for large $k$, we have 
  	$$\tilde{w}_{\rho}(k)\geq c_2 e^{\rho^2/4} (\rho^2(2k+n))^{-\frac{2n+1}{4}} e^{-\psi_{\rho}(k)} $$ where $\psi_{\rho}(k)=4(2k+n) \zeta\left(\frac{\rho^2}{4(2k+n)}\right)^{\frac{1}{2}}-\rho\sqrt{2k}$. It can be checked that for $0<\rho\leq 1$, this function $ \psi_{\rho}(k)$ is decreasing in $k$ whence $\psi_{\rho}(k)\leq c$ for some constant $c $ depending on $\rho$.  So finally we have  
  	\begin{equation}
  	\label{weight}
  	c_2e^{\rho^2/4} (\rho^2(2k+n))^{-\frac{2n+1}{4}} \leq \tilde{w}_{\rho}(k)\leq c_1 e^{\rho^2/4} (\rho^2(2k+n))^{-\frac{2n+1}{4}}. 
  	\end{equation} 
  	  By letting  $m_n=- \frac{1}{8}(2n+1)$, we clearly  see that    $f\in W^s_{\rho}(\R^n)$ if and only if $f\in W^{m_n}_H(\R^n)$ whenever $0<\rho\leq1.$  Here $W^m_H(\R^n)$  denotes the Hermite sobolev spaces.
  	 
  \end{rem}

  In view of the connection between the operators $H$ and $L$, to prove Theorem \ref{hLLL} it suffices to prove the following characterisation for the solution of the extension problem for $H$. Note that the extension problem for the Hermite operator $H$ we are talking about reads as 
  	$$  \Big( -H + \partial_\rho^2 + \frac{1-2s}{\rho} \partial_\rho - \frac{1}{4}\rho^2 \Big) u(x,\rho) = 0,\,\, u(x,0) = f(x) .$$
  	  For $\rho>0$ let $T_{\rho}$ stand for the operator defined for  reasonable  $f$ by  
  	  $$T_{\rho}f(x):=(\frac{1}{2}\rho^2)^{\frac{s-1}{2}}\frac{1}{\Gamma(s)}\sum_{k=0}^{\infty}\Gamma\left(\frac12(2k+n+s+1)\right)W_{-(k+n/2),s/2}( \rho^2/2)P_k f(x).$$ Using similar reasoning as in the case of $L$, we point out that for a tempered distribution $f$,  the above expression makes sense and solves the extension problem for $H$. Moreover, in view of the relation $Q_k=M_{\gamma}^{-1}P_kM_{\gamma}$, we have $T_{\rho}f=M_{\gamma}^{-1}S^1_{\rho}M_{\gamma}^{-1}f.$ Thus, Theorem \ref{hLLL} easily follows from the following theorem.   
  \begin{thm}
  	A solution of the extension problem for $H$ is of the form $u(x,\rho)=T_{\rho}f(x)$ for some   $ f\in W^{m_n}_{H}(\R^n)$ if and only if for every $\rho>0$, $ u(.,\rho)$ extends holomorphically to $\Omega_{\rho/2}$ and satisfies the estimate  
  	\begin{equation}
  	\label{uest}
  	\int_{\Omega_{ \rho/2}}  | u(z,\rho)|^2  w^{2s}_{ \rho/2}(z) dz \leq C  \rho^{n-1/2 }    .
  	\end{equation} for all $0<\rho\leq1.$
  \end{thm} 
  
  \begin{proof}
  	 First suppose $u(x,\rho)=T_{\rho}f(x)$ for some $f$ such that $ f\in W^{m(s)}_{H}(\R^n)$. So clearly 
  	 $$	 u(x,\rho)= \frac{(\frac{1}{2}\rho^2)^{\frac{s-1}{2}}}{\Gamma(s)}\sum_{k=0}^{\infty}\Gamma\left(\frac12(2k+n+s+1)\right)W_{-(k+n/2),s/2}( \rho^2/2)P_k f(x).$$ But the Hermite functions $\Phi_{ \alpha}(x)=H_{\alpha}(x)e^{-|x|^2/2}$ has holomorphic extension to $\C^n.$  Let $\Phi_k(z,w):=\sum_{|\alpha|=k}\Phi_{ \alpha}(z)\Phi_{ \alpha}(w)$. Then using the following estimate (see \cite{T1}) $$|\Phi_k(z,\bar{z})|\leq C(y)(2k+n)^{\frac34(n-1)}e^{2(2k+n)^{\frac12}|y|}$$
  	 along with the asymptotic property \ref{w} we conclude that the series $$\sum_{k=0}^{\infty}\Gamma\left(\frac12(2k+n+s+1)\right)W_{-(k+n/2),s/2}( \rho^2/2)P_k f(z)$$ converges uniformly over compact subsets of $\Omega_{\rho/2}$ and hence defines a holomorphic function in the domain $\Omega_{\rho/2}.$  Now noting that $$\|P_k  u(.,\rho)\|_2^2=\rho^{2s-2}c_s^2\left( \Gamma\left(\frac12(2k+n+s+1)\right)W_{-(k+n/2),s/2}( \rho^2/2)\right)^2\|P_k f\|_2^2,$$ in view of the Proposition \ref{identity} we obtain
  	 $$ \int_{\R^n} \int_{|y|<\rho/2} |u(x+iy,\rho)|^2 w_{\rho/2}^{2s}(x,y) \, dx dy=c_n \rho^{2n} \sum_{k=0}^{\infty} \tilde{w}_{\rho}(k)\|P_kf\|_2^2 $$ 
  	 But in view of \ref{weight}, $$\| f\|_{s,\rho}^2\leq C  e^{\rho^2/4} \rho^{-(2n+1)/2 }\sum_{k=0}^{\infty}(2k+n)^{2m_n}\|P_kf\|_2^2$$ which gives 
  	 $$  \int_{\R^n} \int_{|y|<\rho/2} |u(x+iy,\rho)|^2 w_{\rho/2}^{2s}(x,y) \, dx dy \leq Ce^{\rho^2/4} \rho^{n-1/2 }\|f\|^2_{W^{m_n}_H}$$   proving the first part of the theorem.
  	 
  	 Conversely, let $u(z,\rho) $ be holomorphic on $\Omega_{\rho/2}$ for every $\rho>0 $ satisfying the estimate \ref{uest}.   Let $g_{\rho}$ be a tempered distribution such that
  	 \begin{equation}
  	 \label{grho}
  	 P_ku(.,\rho)=\left(\frac12\rho^2\right)^{\frac{s-1}{2}}\Gamma\left(\frac12(2k+n+s+1)\right)W_{-(k+n/2),s/2}( \rho^2/2)P_kg_{\rho}.
  	 \end{equation}   Now for $0<\rho\leq 1$, using \ref{weight} we have 
  	 $$\|g_{\rho}\|^2_{W^{m_n}_H} \leq C e^{-\rho^2/4}\rho^{ \frac{2n+1}{2}}  \sum_{k=0}^{\infty}\tilde{w}_{\rho}(k)\|P_kg_{\rho}\|_2^2$$ 
  	  Note that using the Proposition \ref{identity}  we obtain 
  	 $$ \int_{\Omega_{\rho/2}}  | u(z,\rho)|^2  w^{2s}_{ \rho/2}(z) dz=c_n\rho^{2n}\sum_{k=0}^{\infty}\tilde{w}_{\rho}(k)\|P_kg_{\rho}\|_2^2. $$     which by the hypothesis yields   $\|g_{\rho}\|^2_{W^{m_n}_H}\leq C$  for all   $0<\rho\leq 1.$ Now by Banach-Alaoglu theorem, we choose a sequence $\{\rho_m\}$ going to $0$ such that $g_{\rho_m}$ converges weakly in $W^{m_n}_H(\R^n)$ as $k\rightarrow \infty$.   Let $f$ be the weak limit in this case.   Now given $\varphi\in S(\R^n)$, we have 
   \begin{align*}
  	  \int_{\R^n}u(x,\rho_m)\varphi(x)dx =\sum_{ k=0  }^{\infty}\int_{\R^n} P_k u(x,\rho_m) \overline{P_k\varphi(x) }dx . 
  	 \end{align*}
  	 But using \ref{grho}, the above integral equals to 
  	 $$ \sum_{k=0 }^{\infty} \left(\frac12\rho_m^2\right)^{\frac{s-1}{2}}\Gamma\left(\frac12(2|\alpha|+n+s+1)\right)W_{-(|\alpha|+n/2),s/2}( \rho_m^2/2) \int_{\R^n} P_kg_{\rho_m}(x) (x) \overline{P_k\varphi(x) }dx$$
  This allows us  to conclude that $u(.,\rho_m)$ converges to $f$ in the sense of distribution. 
  	 Now under the assumption that $u$ solves the extension problem for $H$, the exact same argument as in the beginning of this subsection gives 
  	 $$\widehat{u}(\alpha,\rho)= (\frac{1}{2}\rho^2)^{\frac{s-1}{2}} \left(C_1(|\alpha|)W_{-(|\alpha|+n/2),s/2}(\frac{1}{2}\rho^2)+C_2(|\alpha|)M_{-(|\alpha|+n/2),s/2}(\frac{1}{2}\rho^2)\right) $$ where $\widehat{u}(\alpha,\rho)$ denote the Hermite coefficients. 
  	 But the estimate \ref{uest} gives 
  	 $$\sum_{k=0}^{\infty}\left(C_2(k)M_{-(k+n/2),s/2}(\rho^2/2)\right)^2L^{n+2s-1}_k(- \rho^2/2) \leq C(\rho).$$ But since both   $M_{-(k+n/2),s/2}(\rho^2/2)$ and $L^{n+2s-1}_k(-\rho^2/2)$  have exponential growth in $ k $ (see \ref{m} and \ref{lag})   , the above inequality forces  $C_2(k)$ to be zero .  Now as $u(,\rho_m)$ converges to $f$  and as $\rho_m$ tends to zero  $(\frac{1}{2}\rho_m^2)^{\frac{s-1}{2}}  W_{-(k+n/2),s/2}(\frac{1}{2}\rho_m^2)$ goes to a constant $\Gamma(s)/\Gamma\left(\frac12(2k+n+s+1)\right)$ (see \ref{wasmall}) the theorem follows.  
 \end{proof}

  \section{ Trace Hardy and Hardy's inequality}
  
  \subsection{Trace Hardy inequality} We prove the following trace Hardy inequality only for the operator $ U $ as the case of $ L $ is similar.
  We shall work with the following gradient on $\mathbb{R}^n\times[0,\infty)$  defined by 
  $$ \nabla_{U}u:=\left( 2^{-1/2}\partial_1u, 2^{-1/2} \partial_2u,..., 2^{-1/2}\partial_nu, \partial_\rho u\right).$$
  We	also let $ P_s(\partial_x,\partial_\rho) = \big( -U+\partial_{\rho}^2+\frac{1-2s}{\rho} \partial_{\rho}-\frac{1}{4}\rho^2 \big) $ stand for the extension operator.
  
  \begin{lem} \label{L2.1}
  	Let $u$ and $v$ be two  real valued functions on $\mathbb{R}^n\times[0,\infty) $ such that $u,v\in C^2_{0}\left([0,\infty) , C^2(\mathbb{R}^n)\right)$.  Then for $0<s<1$ we have 
  	\begin{align}\label{2.2}
  	&\int_{0}^{\infty}\int_{\mathbb{R}^n}\left|\nabla_Uu(x,\rho)-\frac{u(x,\rho)}{v(x,\rho)}\nabla_{U}v(x,\rho)\right|^2\rho^{1-2s}d\gamma(x)d\rho\\
  	&=\int_{0}^{\infty}\int_{\mathbb{R}^n}\left(\left|\nabla_Uu(x,\rho)\right|^2+\left(\frac{n}{2}+\frac{1}{4}\rho^2\right)u(x,\rho)^2\right)\rho^{1-2s}d\gamma(x)d\rho \notag\\
  	&+\int_{0}^{\infty}\int_{\mathbb{R}^n}\frac{u(x,\rho)^2}{v(x,\rho)}(P_s(\partial_x,\partial_\rho) v(x,\rho))\rho^{1-2s}d\gamma(x)d\rho +\int_{\mathbb{R}^n}\frac{u(x,0)^2}{v(x,0)}\displaystyle\lim_{\rho\rightarrow0}\left(\rho^{1-2s}\partial_\rho v\right)(x,\rho)d\gamma(x). \notag
  	\end{align}
  \end{lem}
  \begin{proof}
  	For any $1\le j\le n$, we consider the following integral 
  	\begin{equation}
  	\label{th1}
  	\int_{\mathbb{R}^n}\left(\partial_{j}u-\frac{u}{v}\partial_{j}v\right)^2d\gamma(x) =\int_{\mathbb{R}^n}\left((\partial_{j}u)^2-2\frac{u}{v}\partial_{j}u\partial_{j}v+\frac{u^2}{v^2}(\partial_{j}v)^2\right)d\gamma(x).
  	\end{equation}  
  	Now by definition of adjoint wee get
  	\[\int_{\mathbb{R}^n}\frac{u}{v}\partial_{j}u\partial_{j}v d\gamma(x)=\int_{\mathbb{R}^n}u\partial_{j}^*\left(\frac{u}{v}\partial_{j}v\right)d\gamma(x).\]
  	Using the fact that $\partial_{j}^*=2x_j-\partial_{j}$ on $L^2({\gamma})$ we have 
  	\[u\partial_{j}^*\left(\frac{u}{v}\partial_{j}v\right)=2x_j\frac{u^2}{v}\partial_{j}v-\frac{u}{v}\partial_{j}u\partial_{j}v-u^2\partial_{j}\left(\frac{1}{v}\partial_jv\right)\] which together with the above equation yields
  	\begin{align*}
  	2\int_{\mathbb{R}^n}\frac{u}{v}\partial_{j}u\partial_{j}v d\gamma(x)&=\int_{\mathbb{R}^n}\left(2x_j\frac{u^2}{v}\partial_{j}v -u^2\partial_{j}\left(\frac{1}{v}\partial_jv\right)\right)d\gamma(x)\\
  	&=\int_{\mathbb{R}^n}\left(2x_j\frac{u^2}{v}\partial_{j}v  -\frac{u^2}{v}\partial_j^2v+\frac{u^2}{v^2}(\partial_jv)^2\right)d\gamma(x).
  	\end{align*}
  	Hence we have  
  	\begin{align*}
  	\int_{\mathbb{R}^n}\left(\frac{u^2}{v^2}(\partial_jv)^2-2\frac{u}{v}\partial_{j}u\partial_{j}v\right)d\gamma(x)=-\int_{\mathbb{R}^n}\frac{u^2}{v}\partial_{j}^*\partial_{j}v d\gamma(x).
  	\end{align*}
  	Similarly for any $x\in\mathbb{R}^n$ one can obtain
  	\begin{align*}
  	\int_{0}^{\infty}\left(\frac{u^2}{v^2}(\partial_\rho v)^2-2\frac{u}{v}\partial_\rho u\partial_\rho v\right) \rho^{1-2s}d\rho
  	=\int_{0}^{\infty} \frac{u^2}{v^2}\partial_\rho(\rho^{1-2s}\partial_\rho v) d\rho + \frac{u(x,0)^2}{v(x,0)} \displaystyle\lim_{\rho\rightarrow0}\left(\rho^{1-2s}\partial_\rho v\right)(x,\rho).
  	\end{align*}
  	Multiplying both side of \eqref{th1} by $\frac{1}{2}$ and  summing over $j$ we get the required result.
  \end{proof} 
  \begin{thm}[General trace Hardy inequality]
  	Let $0<s<1.$ Suppose $\phi\in L^2(\gamma)$ is a real valued function in the domain of $U_{s}$ such that $\phi^{-1} U_s\phi$ is locally integrable. Then for any    real valued   function $u(x,\rho)$ from the space $C^2_{0}\left([0,\infty), C^2_{b}(\mathbb{R}^n) \right)$   we have 
  	\[\int_{0}^{\infty}\int_{\mathbb{R}^n}\left(\left|\nabla_Uu(x,\rho)\right|^2+\left(\frac{n}{2}+\frac{1}{4}\rho^2\right)u(x,\rho)^2\right)\rho^{1-2s}d\gamma(x)d\rho \ge C_{n,s} \int_{\mathbb{R}^n}u(x,0)^2\frac{L_s\phi(x)}{\phi(x)}d\gamma(x).\]
  \end{thm}
  \begin{proof}
  	To prove this result, we make use of the Lemma \ref{L2.1}. Since the left hand side of \eqref{2.2} is always non-negative, we have for $0<s<1$,
  	\begin{align} \label{2.4}
  	&\int_{0}^{\infty}\int_{\mathbb{R}^n}\left(\left|\nabla_Uu(x,\rho)\right|^2+\left(\frac{n}{2}+\frac{1}{4}\rho^2\right)u(x,\rho)^2\right)\rho^{1-2s}d\gamma(x)d\rho \\
  	&\geq -\int_{0}^{\infty}\int_{\mathbb{R}^n}\frac{u(x,\rho)^2}{v(x,\rho)}(P_s(\partial_x,\partial_\rho)v(x,\rho))\rho^{1-2s}d\gamma(x)d\rho -\int_{\mathbb{R}^n}\frac{u(x,0)^2}{v(x,0)}\displaystyle\lim_{\rho\rightarrow0}\left(\rho^{1-2s}\partial_\rho v\right)(x,\rho)d\gamma(x), \notag
  	\end{align}
  	Now we take $v(x,\rho )=\frac{4^{-s}}{\Gamma(s)}\rho^{2s}\int_{0}^{\infty}  k_{t,s}(\rho) e^{-tL}\phi(x ) dt $. Then $v$ solves the extension problem \ref{ep}, i.e., $ P_s(\partial_x,\partial_\rho)v=0$ and $v(x,0)=\phi(x).$
  	Then from \eqref{2.4}, we have 
  	\begin{align} \label{2.5}
  	&\int_{0}^{\infty}\int_{\mathbb{R}^n}\left(\left|\nabla_Uu(x,\rho)\right|^2+\left(\frac{n}{2}+\frac{1}{4}\rho^2\right)u(x,\rho)^2\right)\rho^{1-2s}d\gamma(x)d\rho \\
  	&\geq -\int_{\mathbb{R}^n}\frac{u(x,0)^2}{v(x,0)}\displaystyle\lim_{\rho\rightarrow0}\left(\rho^{1-2s}\partial_\rho v\right)(x,\rho)d\gamma(x), \notag
  	\end{align}
  	In view of the above, we need to solve the extension problem for $U$ with a given initial condition $\phi.$ Since
  	$-\displaystyle\lim_{\rho\rightarrow0} \rho^{1-2s}\partial_\rho u(x,\rho)=2^{1-2s} \, \frac{\Gamma(1-s)}{\Gamma(s)} \, U_s\phi $  we get the desired inequality.
  \end{proof}

  \begin{cor}
  	\label{preHard}
  	Let $0<s<1$ and $f \in L^2(\gamma)$ with $U_sf \in L^2(\gamma).$ Then we have 
  	\begin{align} \label{2.8}
  	\langle U_sf, f\rangle_{L^2(\gamma)} \geq    \int_{\mathbb{R}^n} \, f^2(x) \, \frac{U_s\phi}{\phi} \, d\gamma(x) \notag
  	\end{align}
  	for any real valued $\phi$ in the domain of $U_s.$
  \end{cor}
  \begin{proof} When $ u $ itself solves the extension problem with initial condition $ f ,$ the proof of Lemma \ref{L2.1} shows that the left hand side of the trace Hardy inequality reduces to $\langle U_sf, f\rangle_{L^2(\gamma)}.$
  \end{proof}

  \subsection{Hardy's inequality from trace Hardy}
  In this subsection we construct a suitable function $\phi$ so that $\frac{L_s\phi}{\phi}$ simplifies. In order to do so, let us quickly recall some basic facts about Laguerre functions.  Let $\alpha > -1$ and $k \in  \mathbb{N}$. The Laguerre polynomial of degree $k$ and type $\alpha$, $L^{\alpha}_k (x)$ is a solution of   the ordinary differential equation
  $$xy^{''} (x) + (\alpha + 1 - x)y'(x) + ky(x) = 0$$
    whose explicit expression is given by 
    \begin{equation}
     \label{1.16}
    L_k^{\alpha}(x)=\sum_{j=0}^k\frac{\Gamma(k+\alpha+1)}{\Gamma(k-j+1)\Gamma(j+\alpha+1)} \, \frac{(-x)^j}{j!}.
    \end{equation}
     Recall that the Laguerre functions of type $(n- 1)$ are given by   
  $$ \varphi_{k }^{n-1}(r)=L_k^{n-1}(\frac{1}{2} r^2) \, e^{-\frac{1}{4} r^2}, r\ge 0.$$
  For more details about such functions we refer the reader to \cite[Chapter 1]{T2}. Now given $s,\rho>0$, we consider the function  $\phi_{s,\rho}$ which is defined in terms of Laguerre polynomials as follows: 
  \begin{equation} \label{3.5}
  \phi_{s,\rho}(x) =\sum_{m=0}^{\infty} C_{2m,\rho}(s) \, L_m^{\frac{n}{2}-1}(|x|^2) \nonumber 
  =e^{\frac{1}{2}|x|^2} \sum_{m=0}^{\infty}  C_{2m,\rho}(s) \, \varphi_m^{\frac{n}{2}-1}(\sqrt{2}|x|), 
  \end{equation}
  where the coefficients are given in terms of $L$ function, by
  $$C_{k,\rho}(s)=\frac{2\pi}{\Gamma ((n/2+1+s)/2)^2 }  \, L\left(\rho, \frac{2k+n}{4}+\frac{1+s}{2},\frac{2k+n}{4}+\frac{1-s}{2}\right).$$
  In the following lemma we show that how these functions are related via the fractional power of the operator under studied.   
 
  \begin{lem}
  	\label{phichange}
  	For $-1<s<1$, we have 
  	\begin{eqnarray}
  	U_s \phi_{-s,\rho}=\frac{\Gamma ((n/2+1+s)/2)^2}{  \Gamma ((n/2+1-s)/2)^2} \, (4\rho)^s \, \phi_{s,\rho}.
  	\end{eqnarray}
  \end{lem}	
  \begin{proof}
  	Let us take two radial functions $g$ and $h$ on $\R^n$ such that 
  	$$g(x)=\pi^{n/2} \, e^{\frac{1}{2}|x|^2} \, h(x),$$
	where $h\in L^2(\R^n).$
  	Moreover we choose $h$ in such a way that, the Laguerre coefficients
  	$$R_m^{\frac{n}{2}-1}(h)=2\frac{\Gamma(m+1)}{\Gamma(m+\frac{n}{2})} \int_0^{\infty} h(r) \, L_m^{\frac{n}{2}-1}(r^2) \, e^{-(1/2)r^2} r^{n-1} \,dr $$
  	are non-zero. By our choice of $h$ and $g$ it is not hard to see that $Q_kg(x)=e^{\frac{|x|^2}{2}} \, P_kh(x).$ Also since $h$ is radial, using a result proved in \cite[Theorem 3.4.1]{T2} we have  
  	\begin{align}
  	\label{rad_h}
  		P_kh(x)=	
  	\begin{cases}
  	0, & \text{if} \, \,  k=2m+1\\
  	R^{\frac{n}{2}-1}_m(h ) \, L^{\frac{n}{2}-1}_m(|x|^2) \, e^{-\frac{1}{2}|x|^2}& \text{if} \, \,  k=2m,
  	\end{cases}
  	\end{align}
  	Now using the definition of  Laguerre function along with the fact that $R^{\frac{n}{2}-1}_m(h )\neq 0$, we see that 
  	\begin{align*}
  	\phi_{s,\rho}(x)&=e^{\frac{1}{2}|x|^2} \sum_{m=0}^{\infty}\frac{2\pi}{\Gamma(\frac{n}{2}+1+s)^2}  \, L(\rho, \frac{4m+n}{4}+\frac{1+s}{2},\frac{4m+n}{4}+\frac{1-s}{2}) \, \varphi_m^{\frac{n}{2}-1}(\sqrt{2}|x|)\\
  	&=e^{\frac{|x|^2}{2}}\sum_{m=0}^{\infty} C_{2m,\rho}(s) \, \left(R^{\frac{n}{2}-1}_m(h )\right)^{-1} \, R^{\frac{n}{2}-1}_m(h ) \,  L^{\frac{n}{2}}_m(|x|^2) \, e^{-\frac{1}{2}|x|^2}.
  	\end{align*}
  	But the observation \eqref{rad_h} and the fact that $Q_kg(x)=e^{\frac{|x|^2}{2}} \, P_kh(x) $ transform the above equation to 
  	\begin{equation}
  	\label{phi_Q}
  	\phi_{s,\rho}(x)=\sum_{k=0}^{\infty} C_{k,\rho}(s) \, \left(R^{\frac{n}{2}-1}_{\lfloor \frac{k}{2}\rfloor}(h )\right)^{-1} Q_kg(x).
  	\end{equation}  Hence using the definition of $U_s$ we have 
  	\begin{equation}
  	\label{lsphi} 
  	U_{s}\phi_{-s,\rho} =\sum_{k=0}^{\infty} C_{k,\rho}(-s) \,2^s\frac{\Gamma (\frac{2k+n}{4}+\frac{1+s}{2})}{\Gamma(\frac{2k+n}{4}+\frac{1-s}{2})} \left(R^{\frac{n}{2}-1}_{\lfloor \frac{k}{2}\rfloor}(h )\right)^{-1}  Q_kg  .\end{equation}
  	But in view of the transformation property \eqref{3.7}, we have 
  	\begin{align} \label{3.9}
  	&C_{k,\rho}(-s)=\frac{2\pi}{\Gamma ((n/2+1-s)/2)^2} L\left(\rho, \frac{2k+n}{4}+\frac{1-s}{2},\frac{2k+n}{4}+\frac{1+s}{2}\right) \notag\\
  	&=\frac{2\pi}{\Gamma ((n/2+1-s)/2)^2} \, (2\rho)^s \, \frac{\Gamma (\frac{2k+n}{4}+\frac{1-s}{2})}{\Gamma(\frac{2k+n}{4}+\frac{1+s}{2})} L\left(\rho, \frac{2k+n}{4}+\frac{1+s}{2},\frac{2k+n}{4}+\frac{1-s}{2}\right) \notag\\
  	&= \frac{\Gamma ((n/2+1+s)/2)^2}{  \Gamma ((n/2+1-s)/2)^2} \, (2\rho)^s \, \frac{\Gamma (\frac{2k+n}{4}+\frac{1-s}{2})}{\Gamma(\frac{2k+n}{4}+\frac{1+s}{2})} \, C_{k,\rho}(s).
  	\end{align}	
  	Hence from \eqref{lsphi} we obtain
  	\[U_s\phi_{-s,\rho}= \frac{\Gamma ((n/2+1+s)/2)^2}{  \Gamma ((n/2+1-s)/2)^2} \, (4\rho)^s\,\phi_{s,\rho} \] proving the lemma.
   
  \end{proof}
  Now in the rest of the section we will calculate $\phi_{s,\rho}$ almost explicitly in terms of Macdonald's function $K_{\nu}$ defined by the integral for $z>0$ 
  \[K_{\nu}(z):=2^{-\nu-1}z^{\nu}\int_{0}^{\infty}e^{-t-\frac{z^2}{4t}}t^{-\nu-1}dt.\]
  \begin{prop}
  	\label{phical}
  	Let $0<s<1$ and $\rho>0$. Then we have  
  	\begin{equation}
  	\phi_{s,\rho}(x)=2\frac{\sqrt{\pi}2^{-(n/2+1+s)/2}}{\sqrt{2\pi}\Gamma((n/2+1+s)/2)}e^{\frac{|x|^2}{2}} (\rho+|x|^2)^{-(n/2+1+s)/2}K_{ (n/2+1+s)/2}(\rho+|x|^2).
  	\end{equation}
  	
  \end{prop}
  \begin{proof}
  	First we note that the following formula proved in \cite[Lemma 3.8]{CRT}
  	\begin{equation}
  	\label{forphical1}
  	\frac{1}{\sqrt{2\pi}}\int_{-\infty}^{\infty}e^{i\lambda t}\left((\rho+r^2)^2+t^2\right)^{(\alpha+2+s)/2}dt=|\lambda|^{\alpha+1}\sum_{k=0}^{\infty}c^{\lambda}_{k,\rho}(s)\varphi_{k}^{\alpha}(\sqrt{(2|\lambda|)}
  	r)
  	\end{equation} where the coefficients $c^{\lambda}_{k,\rho}(s)$ are given by 
  	\[c^{\lambda}_{k,\rho}(s)=\frac{2\pi |\lambda|^{s}}{\Gamma(( \alpha+2+s)/2)^2}  \, L\left(\rho|\lambda|, \frac{4k+2\alpha+2}{4}+\frac{1+s}{2},\frac{4k+2\alpha+2}{4}+\frac{1-s}{2}\right).\]
  	This holds for any $\lambda\neq0$ and $\alpha>-1/2.$ In particular, taking $\alpha=n/2-1$ and $\lambda=1$ in \eqref{forphical1}, we have 
  	\begin{equation}
  	\phi_{s,\rho}(x)=e^{\frac{|x|^2}{2}}\frac{1}{\sqrt{2\pi}}\int_{-\infty}^{\infty}e^{i t}\left((\rho+|x|^2)^2+t^2\right)^{-( n/2+1+s)/2}dt.
  	\end{equation}
  	The right hand side of the above equation can be computed in terms of Macdonald's function $K_{\nu}$.
  	Now we make use of the following formula (see \cite[p. 390]{PBM})
  	\begin{equation}
  	\label{formulak}
  	\int_{0}^{\infty}\frac{\cos br}{(r^2+z^2)^{\delta}}dr=\left(\frac{2z}{b}\right)^{1/2-\delta}\frac{\sqrt{\pi}}{\Gamma(\delta)}K_{1/2-\delta}(bz)\end{equation} which is valid for $b>0$ and $\Re\delta,\Re z>0$. This gives 
  	\begin{align}
  	\int_{-\infty}^{\infty}&e^{i t}\left((\rho+|x|^2)^2+t^2\right)^{-( n/2+1+s)/2}dt \nonumber\\ &=2\frac{\sqrt{\pi}2^{-(n/2+1+s)/2}}{\Gamma((n/2+1+s)/2)}(\rho+|x|^2)^{-(n/2+1+s)/2}K_{-(n/2+1+s)/2}(\rho+|x|^2)
  	\end{align}
  	Now using the fact that $K_{\nu}=K_{-\nu}$ we obtain
  	\begin{equation}
  	\label{phiexp}
  	\phi_{s,\rho}(x)=2\frac{\sqrt{\pi}2^{-(n/2+1+s)/2}}{\sqrt{2\pi}\Gamma((n/2+1+s)/2)}e^{\frac{|x|^2}{2}} (\rho+|x|^2)^{-(n/2+1+s)/2}K_{ (n/2+1+s)/2}(\rho+|x|^2)
  	\end{equation}
  	 proving the proposition.
  \end{proof}
We are now ready to prove Theorem \ref{h-u_s}. For the convenience of the reader we state the theorem here as well.
  \begin{thm} 
  Let $0<s<1$. Assume that $f\in L^2(\gamma)$ such that $U_{s}f\in L^2(\gamma)$.  Then   for every $ \rho > 0 $ we have  
  \[\langle U_sf,f\rangle_{L^2(\gamma)} \ge (2\rho)^s \frac{\Gamma\left(\frac{n/2+1+s}{2}\right) }{\Gamma\left(\frac{n/2+1-s}{2}\right) } \int_{\R^n}\frac{f(x)^2}{(\rho+|x|^2)^s}
  w_s(\rho+|x|^2)  \,\,d\gamma(x) \]
  for an explicit  $ w_s(t) \geq 1.$ The inequality is sharp and equality is attained for $ f(x) = \phi_{-s,\rho}(x).$ 
  \end{thm}
  \begin{proof}
  	 Taking $\phi=\phi_{-s,\rho}$ in \ref{phical}, in view of Lemma \ref{phichange} we have $$\frac{U_s\phi }{\phi}=\frac{\Gamma ((n/2+1+s)/2)^2}{  \Gamma ((n/2+1-s)/2)^2} \, (4\rho)^s\frac{\phi_{s,\rho}}{\phi_{-s,\rho}}.$$ Now we use Proposition \ref{phical} to simplify the right hand side of the above equation. Note that
  	 \begin{align}
  	 \label{phioverphi}
  	 \frac{\phi_{s,\rho}}{\phi_{-s,\rho}} =   \frac{\Gamma\left(\frac{n/2+1+s}{2}\right)}{\Gamma\left(\frac{n/2+1-s}{2}\right)}2^{-s}(\rho+|x|^2)^{-s} \frac{K_{ (n/2+1+s)/2}(\rho+|x|^2)}{K_{ (n/2+1-s)/2}(\rho+|x|^2)}.
  	 \end{align}
  	 Let $$w_{s}(t):=\frac{K_{ (n/2+1+s)/2}(t)}{K_{ (n/2+1-s)/2}(t)},\,\,\,  ~t>0.$$ Now using the fact that for $ t>0$, $K_{\nu}(t)$ is increasing function of $\nu$ we note that $w_s(t)\ge 1, $ for all $t>0$.
  	\[\frac{U_s\phi}{\phi}=2^s\rho^s \frac{\Gamma\left(\frac{n/2+1-s}{2}\right)}{\Gamma\left(\frac{n/2+1+s}{2}\right)} (\rho+|x|^2)^{-s}w_s(\rho+|x|^2).\] Hence the required inequality  follows from Corollary \ref{preHard}.  
  	
  	To see the equality holds for $f(x)=\phi_{-s,\rho}(x)$, using Lemma \ref{phichange} we note that 
  	\begin{align*}
  	\langle U_s\phi_{-s,\rho}, \phi_{-s,\rho}\rangle_{L^2(\gamma)}=\frac{\Gamma ((n/2+1+s)/2)^2}{  \Gamma ((n/2+1-s)/2)^2} \, (4\rho)^s \int_{\R^n}\phi_{-s,\rho}(x)^2\frac{\phi_{s,\rho}(x)}{\phi_{-s,\rho}(x)}d\gamma(x)
  	\end{align*}
  	But the equation \ref{phioverphi} allow us  to write the above as 
  	$$\langle U_s\phi_{-s,\rho}, \phi_{-s,\rho}\rangle_{L^2(\gamma)}=(2\rho)^s\int_{\R^n}\frac{\phi_{-s,\rho}(x)^2}{(\rho+|x|^2)^s}w_s(\rho+|x|^2)d\gamma(x)$$ which proves equality case.
  \end{proof}

  \section{Isometry property for the solution of the extension problem}
  In this section we prove an isometry property of the solution operator associated to the extension problem for Ornstein-Uhlenbeck operator under consideration. Such a property has been studied in the context of extension problem for Laplacian on $\R^n$ and for sublaplacian on $\mathbb{H}^n$ in M\"ollers et al \cite{MOZ}. See also the work of Roncal-Thangavelu \cite{RT2} where they proved similar result in the context of $H$-type groups.

  We consider the Gaussian sobolev space  $\mathcal{H}^s_{\gamma}(\mathbb{R}^n)$ defined via the relation $f\in\mathcal{H}^s_{\gamma}(\mathbb{R}^n) $ if and only if  $L_{s/2}f\in L^2(\gamma)$, where $L_{s/2}$ is the fractional power under consideration. Instead of $ \| L_{s/2}f \|_2 $  we use  the equivalent norm for  this space which is given by  
  \[\|f\|_{(s)}^2:=\langle L_sf,f\rangle_{L^2(\gamma)}=\sum_{\alpha\in\mathbb{N}^n} 2^s\frac{\Gamma\left( \frac{2|\alpha|+n}{2}+\frac{1+s}{2}\right)}{\Gamma\left( \frac{2|\alpha|+n}{2}+\frac{1-s}{2}\right)}|\langle f,H_{\alpha}\rangle_{L^2(\gamma)}|^2 .\]
  Recall that $ H_\alpha $ are the normalised Hermite polynomials on $ \R^n $ forming an orthonormal basis for $ L^2(\gamma).$
  As the solution of the extension problem \ref{ep} is a function of $\rho^2$, it can be thought of as a function of $ (x,y) \in \mathbb{R}^{n+2}$ that is radial in $ y.$ Thus it makes sense  to define $ P_sf(x,y) = u(x, \sqrt{2}|y|) $ where $ u (x,\rho) $ is the solution of the extension problem  \ref{ep} given by \ref{epsol}.  We can now consider $ P_sf(x,y) $ as an element of $ L^2(\R^{n+2}, \gamma) .$   For $(\alpha,j)\in\mathbb{N}^n\times\mathbb{N}^2$ we let   $$H_{\alpha,j}(x,y):=   H_{\alpha}(x)H_{j}(y),\ (x,y)\in\mathbb{R}^n\times\mathbb{R}^2$$ where $H_j $'s are two dimensional   Hermite polynomials . Then $ P_sf(x,y) $ can be expanded in terms of $ H_{\alpha,j}(x,y).$ We will show that $ P_s $ takes $ \mathcal{H}_\gamma^s(\R^n) $  into $ \mathcal{H}^{s+1}(\R^{n+2}).$ 
  We equip $ \mathcal{H}_\gamma^{s+1}(\R^{n+2}) $ with a different but equivalent norm.  For  $u \in  \mathcal{H}_\gamma^{s+1}(\R^{n+2}) $  we define  
  \[ \|u\|_{(1,s)}^2=\displaystyle\sum_{(\alpha,j)\in\mathbb{N}^n\times\mathbb{N}^2}2^{s+1}\frac{\Gamma\left( \frac{2|\alpha|+2|j|+n+1}{2}+\frac{1+(1+s)}{2}\right)}{\Gamma\left( \frac{2|\alpha|+2|j|+n+1}{2}+\frac{1-(1+s)}{2}\right)}|\langle u^j,H_{\alpha }\rangle_{L^2(\gamma)}|^2 \] where for any $j\in \mathbb{N}^2 $ we have let 
  \[u^j(x):=\int_{\mathbb{R}^2 }u(x,y)H_j(y)e^{-|y|^2/2}dy.\] Equipped with this norm we denote the space $ \mathcal{H}_\gamma^{s+1}(\R^{n+2}) $  by $\tilde{\mathcal{H}}^{s+1}_{\gamma}(\mathbb{R}^{n+2}) .$ 
  \begin{thm}
  	For $0<s<n$, $P_s:\mathcal{H}^s_{\gamma}(\mathbb{R}^n)\rightarrow\tilde{\mathcal{H}}^{s+1}_{\gamma}(\mathbb{R}^{n+2})$ is a constant multiple of  an isometry, i.e. $\|P_sf\|_{(1,s)}=C_{n,s}\|f\|_{(s)}$ for all $ f \in \mathcal{H}_\gamma^{s}(\R^{n}). $ 
  \end{thm}
  
  \begin{proof}
  	We have  
  	$$P_sf(x,y)=\sum_{k=0}^{\infty} \frac{2^{-s}}{\Gamma(s)} \, (\sqrt{2}|y|)^{2s} \, L\left(\frac{1}{2} \, |y|^2,\frac{2k+n}{2}+\frac{1+s}{2},\frac{2k+n}{2}+\frac{1-s}{2}\right) \, Q_kf.$$ Now from \ref{Ls1}  we note that
  	\begin{align*}
  	&P_sf(x,y)=T_{-s,\sqrt{2}|y|} (L_sf)(x)\\
  	&=\sum_{k=0}^{\infty} \frac{4^{s}}{\Gamma(-s)} \,   \, L\left(\frac{1}{2} \, y^2,\frac{2k+n}{2}+\frac{1-s}{2},\frac{2k+n}{2}+\frac{1+s}{2}\right) \,   \, \frac{\Gamma(\frac{2k+n}{2}+\frac{1+s}{2})}{\Gamma(\frac{2k+n}{2}+\frac{1-s}{2})} \, Q_kf.
  	\end{align*}
  	Now writing $a:=\frac{2k+n}{2}+\frac{1+s}{2}$ and $b:=\frac{2k+n}{2}+\frac{1-s}{2}$, we expand  $ L(y^2/2,a,b)$ in terms of Hermite polynomials. In order to do that we use Mehler's formula (see \cite[Chapter 1]{WU}) for 2- dimensional normalised Hermite polynomials:
  	\begin{align*}
  	\sum_{j\in\mathbb{N}^2} H_{j}(x) H_{j}(y)r^{|j|}=(1-r^2)^{-1 }e^{-\frac{r^2(|x|^2+|y|^2)}{1-r^2}-\frac{2rx.y}{1-r^2}}.
  	\end{align*}  
  	In view of the definition of the  $L$ function we have 
  	\[L(|y|^2/2,a,b)=e^{-|y|^2/2}\int_{0}^{\infty}e^{-t|y|^2 }t^{a-1}(1+t)^{-b}dt.\]
  	Now taking $r^2=\frac{t}{1+t}$ in   the above Mehler's formula, we have 
  	\begin{align*}
  	e^{-t|y|^2}=(1+t)^{-1}\sum_{j\in\mathbb{N}^2}H_{j}(0) H_{j}(y)\left(\frac{t}{1+t}\right)^{|j|/2}
  	\end{align*}
  	which yields
  	\begin{align*} 
  	L(y^2/2,a,b)&=e^{-|y|^2/2}\sum_{j\in\mathbb{N}^2}H_{j}(0) H_{j}(y)\int_{0}^{\infty} t^{a+|j|/2-1}(1+t)^{-b-|j|/2-1}dt\\
  	&=e^{-|y|^2/2}\sum_{j\in\mathbb{N}^2}H_{j}(0) H_{j}(y)\frac{\Gamma(a+|j|/2)\Gamma(b-a+1)}{\Gamma(b+|j|/2+1)}. 
  	\end{align*}
  	Here the second equality follows from the following formula: 
  	\[\int_{0}^{\infty}(1+t)^{-b}t^{a-1}dt=\frac{\Gamma(a)\Gamma(b-a)}{\Gamma(b)}.\]
  	Finally writing $P_sf(x,y)=v(x,y)$ and  using the above observations we have  
  	\begin{align}
  	v(x,y)=c_s e^{-|y|^2/2}\!\!\!\!\!\!\!\!\displaystyle\sum_{(\alpha,j)\in \mathbb{N}^n\times \mathbb{N}^2}\!\!\!\!\!\!\!\! H_{j}(0) H_{j}(y)\frac{\Gamma(a+|j|/2)\Gamma(b-a+1 )}{\Gamma(b+|j|/2+1)}\frac{\Gamma(\frac{2|\alpha|+n}{2}+\frac{1+s}{2})}{\Gamma(\frac{2|\alpha|+n}{2}+\frac{1-s}{2})} \langle f, H_{\alpha }\rangle_{L^2(\gamma)}H_{\alpha}(x) 
  	\end{align}
  	where $c_s:=\frac{4^s}{\Gamma(-s)}.$  
  	Now note that for any $j\in\mathbb{N}^2$ we obtain
  	\begin{align*}
  	v^j(x)=c_{ s}\sum_{ \alpha  \in \mathbb{N}^n } H_{j}(0) \frac{\Gamma(a+|j|/2)\Gamma(b-a+1)}{\Gamma(b+|j|/2+1)}\frac{\Gamma(\frac{2|\alpha|+n}{2}+\frac{1+s}{2})}{\Gamma(\frac{2|\alpha|+n}{2}+\frac{1-s}{2})} \langle f, H_{\alpha }\rangle_{L^2(\gamma)}H_{\alpha}(x) 
  	\end{align*}
  	which yields
  	\begin{align*}
  	\langle v^j, H_{\alpha }\rangle_{L^2(\gamma)}=c_s H_{j}(0) \frac{\Gamma(a+|j|/2)\Gamma(b-a+1)}{\Gamma(b+|j|/2+1)}\frac{\Gamma(\frac{2|\alpha|+n}{2}+\frac{1+s}{2})}{\Gamma(\frac{2|\alpha|+n}{2}+\frac{1-s}{2})} \langle f, H_{\alpha }\rangle_{L^2(\gamma)}.
  	\end{align*}
  	As shown in \cite{WU}, for any $k \in \mathbb{N}$ and for one dimensional Hermite polynomials we have 
  	\begin{align*}
  	H_{2k+1}(0)=0 \ ~\text{and}\ ~ (H_{2k}(0))^2=\frac{2^{-2k}\Gamma(2k+1)}{\Gamma(k+1)^2}.
  	\end{align*}
  	But making use of the formula $\Gamma(2z)=(2\pi)^{-1/2}2^{2z-1/2}\Gamma{(z)}\Gamma(z+1/2)$ we obtain 
  	\begin{align*}
  	(H_{2k}(0))^2=\frac{1}{\sqrt{ \pi}}\frac{\Gamma(k+1/2)}{\Gamma(k+1)}.
  	\end{align*}
  	Hence for $j=(j_1,j_2)\in \mathbb{N}^n$ we have 
  	\[(H_{2j}(0))^2=\frac{1}{ \pi}\frac{\Gamma(j_1+1/2)\Gamma(j_2+1/2)}{\Gamma(j_1+1)\Gamma(j_2+1)}.\]
  	With these things in hand we proceed to calculate $ \|v\|_{(1,s)}^2 $ which is given by a constant multiple of   
  	$$ \sum_{k=0}^{\infty}\sum_{ j \in \mathbb{N}^2}     \frac{\Gamma\left( \frac{2k+2|j|+n+1}{2}+\frac{1+(1+s)}{2}\right)}{\Gamma\left( \frac{2k+2|j|+n+1}{2}+\frac{1-(1+s)}{2}\right)}\left|H_{j}(0) \frac{\Gamma(a+|j|/2)\Gamma(b-a+1)}{\Gamma(b+|j|/2+1)}\frac{\Gamma(\frac{2k +n}{2}+\frac{1+s}{2})}{\Gamma(\frac{2k+n}{2}+\frac{1-s}{2})} \right|^2 \|Q_kf\|^2
  	$$ where $\|Q_kf\|^2=\displaystyle
  	\sum_{|\alpha|=k}\left|\langle f,H_{\alpha}\rangle_{L^2(\gamma)}\right|^2$.
  	Now we have already noted the fact that $H_{2k+1}(0)=0$.  In what follows both $ j_1$ and $ j_2 $ should be even. Using the values of $a$ and $b$ we have 
  	\begin{align*}
  	&\sum_{j=(j_1,j_2)\in\mathbb{N}^2}\frac{\Gamma\left( \frac{2k+2|j|+n+1}{2}+\frac{1+(1+s)}{2}\right)}{\Gamma\left( \frac{2k+2|j|+n+1}{2}+\frac{1-(1+s)}{2}\right)}\left(H_{j}(0) \frac{\Gamma(a+|j|/2)\Gamma(b-a+1)}{\Gamma(b+|j|/2+1)}\right)^2\\ &=\frac{\Gamma(s+1)^2}{\pi}\sum_{j\in\mathbb{N}^2}\frac{\Gamma\left( \frac{2k+2|j|+n+1}{2}+\frac{1-(1+s)}{2}\right)}{\Gamma\left( \frac{2k+2|j|+n+1}{2}+\frac{1+(1+s)}{2}\right)}\frac{\Gamma(j_1+1/2)\Gamma(j_2+1/2)}{\Gamma(j_1+1)\Gamma(j_2+1)}.
  	\end{align*}
  	In order to simplify this further we make use of some properties of Hypergeometric functions. We start with recalling that 
  	\begin{align*}
  	F(\delta, \beta, \eta,z)=\sum_{k=0}^{\infty}\frac{(\delta )_k(\beta)_k}{(\eta)_kk!}z^k=\frac{\Gamma(\eta)}{\Gamma(\delta)\Gamma(\beta)}\sum_{k=0}^{\infty}\frac{\Gamma(\delta+k)\Gamma(\beta+k)}{\Gamma(\eta+k)\Gamma(k+1)}z^k.
  	\end{align*}  Here we will be using the following property proved in \cite{OM}:
  	\begin{align}
  	&\frac{\Gamma(\eta)\Gamma(\eta-\delta-\beta)}{\Gamma(\eta-\delta)\Gamma(\eta-\beta)}=F(\delta,\beta,\eta,1)=\frac{\Gamma(\eta)}{\Gamma(\delta)\Gamma(\beta)}\sum_{k=0}^{\infty}\frac{\Gamma(\delta+k)\Gamma(\beta+k)}{\Gamma(\eta+k)\Gamma(k+1)}.\nonumber
  	\end{align}
  	That is, 
  	\begin{align}
  	\sum_{k=0}^{\infty}\frac{\Gamma(\delta+k)\Gamma(\beta+k)}{\Gamma(\eta+k)\Gamma(k+1)}= \frac{\Gamma(\beta )\Gamma(\eta-\delta-\beta)\Gamma(\delta)}{\Gamma(\eta-\delta)\Gamma(\eta-\beta)}\ \text{provided}\ \Re(\eta-\delta-\beta)>0.
  	\end{align}
  	Taking $\delta=\frac{2k+2j_2+n+1-s}{2},\ \beta=\frac{1}{2} $ and $\eta=\frac{2k+2j_2+n+3+s}{2}$ in the above formula we have 
  	\begin{align*} 
  	\sum_{j_1=0}^{\infty}\frac{\Gamma(\delta+j_1)\Gamma(\beta+j_1)}{\Gamma(\eta+j_1)\Gamma(j_1+1)}= \frac{\Gamma(s+1/2)\Gamma(1/2)}{\Gamma(s+1 )} \frac{\Gamma(\frac{2k+2j_2+n+1-s}{2})}{\Gamma(\frac{2k+2j_2+n+2+s}{2})}.
  	\end{align*}
  	This gives 
  	\begin{align*}
  	&\sum_{j\in\mathbb{N}^2}\frac{\Gamma\left( \frac{2k+2|j|+n+1}{2}+\frac{1-(1+s)}{2}\right)}{\Gamma\left( \frac{2k+2|j|+n+1}{2}+\frac{1+(1+s)}{2}\right)}\frac{\Gamma(j_1+1/2)\Gamma(j_2+1/2)}{\Gamma(j_1+1)\Gamma(j_2+1)}\\&=  \frac{\Gamma(s+1/2)\Gamma(1/2)}{\Gamma(s+1)} \sum_{j_2=0}^{\infty}\frac{\Gamma(\frac{2k +n+1-s}{2}+j_2)\Gamma(j_2+1/2)}{\Gamma(\frac{2k +n+2+s}{2}+j_2)\Gamma(j_2+1) }\\
  	&= \frac{\Gamma(s+1/2)\Gamma(1/2)}{\Gamma(s+1)} \frac{\Gamma(\frac{2k +n+1-s}{2})}{\Gamma(\frac{2k +n+1+s}{2})}\frac{\Gamma(s)\Gamma(1/2)}{\Gamma(s+1/2)}\\
  	&=\frac{\Gamma(1/2)^2}{s}\frac{\Gamma(\frac{2k +n+1-s}{2})}{\Gamma(\frac{2k +n+1+s}{2})}.
  	\end{align*}
  	Therefore, we have 
  	\begin{align*}
  	\|v\|_{(1,s)}^2& =c_s^2\Gamma(s+1)^2\frac{2\Gamma(1/2)^2}{\pi s} \displaystyle\sum_{k=0}^{\infty}2^{s}\frac{\Gamma(\frac{2k +n+1-s}{2})}{\Gamma(\frac{2k +n+1+s}{2})}\|Q_kf\|^2\\ &=c_{n,s}\displaystyle\sum_{  \alpha \in\mathbb{N}^n }2^{s}\frac{\Gamma(\frac{2|\alpha| +n+1-s}{2})}{\Gamma(\frac{2|\alpha| +n+1+s}{2})}|\langle f, H_{\alpha }\rangle_{L^2(\gamma)}|^2 \\
  	&=c_{n,s} \|f\|^2_{(s)}.
  	\end{align*}
  	This completes the proof. 
  \end{proof}
  
  \section{Hardy-Littlewood-Sobolev inequality for $H_s$}	In this section we are interested in Hardy-Littlewood-Sobolev inequality for the fractional powers $ H_s.$ For the Laplacian on $ \R^n $ and sublaplacian on $ \mathbb{H}^n$ such inequalities with sharp constants are known. Let us  recall the inequality for the sublaplacian $ \mathcal{L} $ on $\mathbb{H}^n.$   Letting $q=\frac{2(n+1)}{n+1-s}, $  the Hardy-Littlewood-Sobolev inequality for $\mathcal{L}_s$  (see \cite{BFM}, \cite{FL}) reads as
\begin{eqnarray} \label{HLSH}
 \frac{\Gamma (\frac{1+n+s}{2})^2}{\Gamma (\frac{1+n-s}{2})^2} \, w_{2n+1}^{\frac{s}{n+1}} \, \left(\int_{\mathbb{H}^n} |g(z,w)|^q dz dw\right)^{\frac{2}{q}}  \leq  \langle \mathcal{L}_s g,g\rangle . 
\end{eqnarray}  
 
 
We first find an integral representation of $H_{-s}$ using the  integral representation of fractional power of the sublaplacian, $\mathcal{L}_{-s}$. The integral kernel of $\mathcal{L}_{-s}$ is given by $c_{n,s}|(z,t)|^{-Q+2s}$ as shown by Roncal-Thangavelu in \cite{RT1}. Here $|(z,t)|:=(|z|^4+t^2)^{1/4}$ denotes the Koranyi norm on the Heisenberg group and $Q=2n+2$ is its homogeneous dimension. We consider the Schrodinger representation $\pi_{\lambda}$ of $\mathbb{H}^n$ whose action on the representation space  $L^2(\mathbb{R}^n)$ is given by 
\[\pi_{\lambda}(z,t)\phi(\xi)=e^{i\lambda t}e^{i\lambda\left(x\xi+\frac{1}{2}x.y\right)}\phi(\xi+y).\]
The Fourier transform of a function $ f \in L^1(\mathbb{H}^n) $ is the operator valued function defined on the set of all nonzero reals, $ \R^\ast $ given by
$$  \hat{f}(\lambda) = \int_{\mathbb{H}^n}   f(z,t)  \pi_\lambda(z,t) dz dt .$$ The action of the Fourier transform on function of the form $\mathcal{L}$ is well-known and is given by $\widehat{\mathcal{L}f}(\lambda)=\widehat{f}(\lambda)H(\lambda)$ where $H(\lambda) $ is the scaled Hermite operator.  In view of this,  it can be easily checked that 
 \begin{equation}
 \label{HLrel}
d\pi_{\lambda}(m(\mathcal{L}))=m(H(\lambda))
\end{equation} where   $d\pi_{\lambda}$ stands for the derived representation corresponding to $\pi_{\lambda}$.  Recall that the fractional power $\mathcal{L}_{-s}$ is defined as follows (see Roncal-Thangavelu \cite{RT1})  

\[\mathcal{L}_{-s}f(z,t):=(2\pi)^{-n-1}\int_{-\infty}^{\infty
}\left(\sum_{k=0}^{\infty}(2|\lambda|)^{-s}\frac{\Gamma (\frac{2k+n}{2}+\frac{1-s}{2})}{\Gamma(\frac{2k+n}{2}+\frac{1+s}{2})}f^{\lambda}\ast_{\lambda}\varphi^{\lambda}_{k}(z)\right)e^{-i\lambda t}|\lambda|^nd\lambda.\] So we have $d\pi_{\lambda}( \mathcal{L}_{-s})= H(\lambda)_{-s} .$ In particular for $\lambda=1$, using spectral decomposition we have  
\[H_{-s}f=\sum_{k=0}^{\infty}2^{-s}\frac{\Gamma (\frac{2k+n}{2}+\frac{1-s}{2})}{\Gamma(\frac{2k+n}{2}+\frac{1+s}{2})}P_kf.\] 
Now it is not hard to see that 
\begin{equation}
H_{-s}(fe^{-|.|^2/2})(x)=e^{-|x|^2/2}\sum_{k=0}^{\infty}2^{-s}\frac{\Gamma (\frac{2k+n}{2}+\frac{1-s}{2})}{\Gamma(\frac{2k+n}{2}+\frac{1+s}{2})}Q_kf(x).
\end{equation}
Hence from the definition of $L_{-s}$ we have 
\begin{equation}
\label{relHL}
H_{-s}(fe^{-|.|^2/2})(x)= e^{-|x|^2/2}L_{-s}f(x).
\end{equation}
In this section, we prove an analogue of \eqref{HLSH}  for the operator $ H_{-s}.$ We first study $L^p-L^q$ mapping properties  of the operator $H_{-s}.$

In view of this relation \ref{HLrel}  we have 
$$H_{-s}f(\xi)=c_{n,s}\int_{\mathbb{H}^n}|(z,t)|^{-Q+2s}\pi_{1}(z,t)f(\xi)dzdt .$$

Using the definition of $\pi_{1}$ and writting $z=x+iy$, we obtain 
\begin{eqnarray*} 
	H_{-s}f(\xi)\!\!\!&&=c_{n,s}\int_{\mathbb{H}^n}((|x|^2+|y|^2)^2+t^2)^{-\frac{n+1-s}{2} }e^{i\lambda t}e^{i\lambda\left(x\xi+\frac{1}{2}x.y\right)}f(\xi+y)dx \, dy \, dt\\
	&&=c_{n,s}\int_{\mathbb{H}^n}((|x|^2+|\eta-\xi|^2)^2+t^2)^{-\frac{n+1-s}{2} }e^{i \lambda t}e^{i\frac{\lambda}{2}\left(x\xi+x.\eta\right)}f(\eta)dx \, d\eta \, dt\\
	&&= \int_{\R^n} K_H^s(\xi,\eta) \, f(\eta) \, d\eta,
\end{eqnarray*} 
where  the kernel $ K_H^s $ is defined by 
\begin{equation}
\label{HLSker}
K_H^s(\xi,\eta)=c_{n,s}\int_{\mathbb{R}^n\times \R}((|x|^2+|\eta-\xi|^2)^2+t^2)^{-\frac{n+1-s}{2} }e^{i \lambda t}e^{i\frac{\lambda}{2}\left(x\xi+x.\eta\right)}dx  \, dt.\end{equation}
After taking the modulus and then a change of variables leads to 
$$ |K_H^s(\xi,\eta)|\leq c_{n,s}\int_{\mathbb{R}^n}(|x|^2+|\eta-\xi|^2)^{-n+s } \, dx.$$
Now again a change of variable $x\rightarrow x|\xi-\eta|$ yields  \begin{equation}
|K_{H}^s(\xi,\eta)|\leq C_{n,s} |\xi-\eta|^{-n+2s}.
\end{equation}  Now it is a routine matter to check the following $L^p-L^q$ boundedness property, see e.g., \cite[Theorem 6.1.3 ]{LG}. In fact,  for $1<p<q<\infty$ with $\frac{1}{p}-\frac{1}{q}=\frac{2s}{n} $ we get   
\begin{equation}
\|H_{-s}f\|_{L^q} \leq C_{n,s}(p) \, \|f\|_{L^p}.
\end{equation}
 
Nevertheless, in the following theorem we obtain better estimate for the kernel improving the above mentioned $L^p-L^q$ estimates. 

\begin{thm} \label{thm5.1}
 For any $1\le p\leq q<\infty$ with $\frac1p-\frac1q\leq1$ there exists a constant $C_{n,s}(p)$ such that for all $f \in L^p(\R^n)$, the inequality $\|H_{-s}f\|_{L^q} \leq C_{n,s}(p) \, \|f\|_{L^p}$ holds.
\end{thm}
\begin{proof} 
		 In view  of the formula stated in \ref{formulak}, from \ref{HLSker} we have 
		 \begin{equation}
		 K^s_{H}(\xi,\eta):=2c_{n,s}\frac{\sqrt{\pi}2^{-(n/2+1+s)/2}}{\Gamma(\frac{n+1-s}{2})}\int_{\mathbb{R}^n} (|x|^2+|\eta-\xi|^2)^{-\frac{n+1-s}{2}}K_{-\frac{n+1-s}{2}}(|x|^2+|\eta-\xi|^2)e^{  \frac{i}{2}x.(\eta+\xi) } dx.   
		 \end{equation} 
		  Now we use the integral representation of $ K_\nu $ to simplify the above integral giving the kernel.
		 	$$ K_\nu(z)  = 2^{-\nu-1} z^\nu \int_0^\infty  e^{-t-\frac{z^2}{4t}} t^{-\nu-1} dt.$$
		 	A simple change of variables shows that
		 	$$ z^\nu K_\nu(z) = 2^{\nu-1}   \int_0^\infty  e^{-t-\frac{z^2}{4t}} t^{\nu-1} dt =  z^\nu K_{-\nu}(z).$$
		 	Thus
		 	$$ (|x|^2+|\eta-\xi|^2)^{-\frac{n+1-s}{2}}K_{-\frac{n+1-s}{2}}(|x|^2+|\eta-\xi|^2) = (|x|^2+|\eta-\xi|^2)^{-\frac{n+1-s}{2}}K_{\frac{n+1-s}{2}}(|x|^2+|\eta-\xi|^2)$$ leading to the formula
		 	$$ (|x|^2+|\eta-\xi|^2)^{-\frac{n+1-s}{2}}K_{-\frac{n+1-s}{2}}(|x|^2+|\eta-\xi|^2) = 2^{-\nu-1}  \int_0^\infty  e^{-t-\frac{z^2}{4t}} t^{-\nu-1} dt$$ 
		 	where $ \nu = \frac{n+1-s}{2}$ and $ z = (|x|^2+|\xi-\eta|^2) .$ Writing $a:=\frac{1}{2}(\xi+\eta)$, we   estimate the integral
		 	$$ \int_{\R^n}  e^{-\frac{1}{4t} (|x|^2+r^2)^2} e^{i x \cdot a} dx  $$ where we have let $r=|\xi-\eta|.$  
		 First note that
		   $$ \int_{\R^n}  e^{-\frac{1}{4t} (|x|^2+r^2)^2} e^{i x \cdot  a} dx = e^{-\frac{1}{4t}r^4} \int_{\R^n} e^{i x \cdot a} e^{-\frac{2}{4t} r^2 |x|^2} e^{-\frac{1}{4t}|x|^4} dx.$$
		    Let $ \varphi$ stand for the Fourier transform of the function $ e^{-\frac{1}{4}|x|^4}.$ So the above integral is bounded by
		 	$$ e^{-\frac{1}{4t}r^4} (t/r^2)^{n/2} t^{n/4} \int_{\R^n}  \varphi(t^{1/4} (a-y)) e^{-\frac{t}{2r^2}|y|^2} dy$$ which is bounded by ( after making a change of variables and using $ |\varphi(\xi)| \leq C$)
		 	$$  e^{-\frac{1}{4t}r^4}  t^{n/4} $$ 
		 	and $ K^s_H(\xi,\eta) $ is bounded by
		 	$$   \int_0^\infty  e^{-t} e^{-\frac{r^4}{4t}} t^{-\frac{n+2-2s}{4}-1} dt = r^{- \frac{n+2-2s}{2}} K_{(n+2-2s)/4}(r^2).$$ Finally we have 
		 	\begin{equation}
		 	|K^s_{H}(\xi,\eta)|\leq C |\xi-\eta|^{-(n+2-2s)/2}K_{\frac{n+2-2s}{4}}(|\xi-\eta|^2)=:G(\xi-\eta).
		 	\end{equation}
		 	Now we see that 
		 	\begin{equation}
		 	|H_{-s}f(\xi)|\leq C |f|\ast G(\xi),\ \forall \xi\in\mathbb{R}^n.
		 	\end{equation} 
		 		Now note that for $r\ge 1$, integrating in polar co-ordinates, we have 
		 	\begin{align*}
		 	\int_{\mathbb{R}^n}G(x)^rdx= c_n\int_{0}^{\infty} \left(t^{-(n+2-2s)/2}K_{\frac{n+2-2s}{4}}(t^2)\right)^rt^{n-1}dt.
		 	\end{align*}
		 	Using the facts that $K_{\nu}(z)\sim z^{-1/2}e^{-z}$ for large $z$ and near the origin $z^{-\nu}K_{\nu}(z)$ is bounded, we conclude that the above integral is finite. Now in view of the Young's inequality we have 
		 	\begin{equation}
		 	\||f|\ast G\|_{q}\leq \|f\|_{p}\|G\|_{r},\ \text{where}\ \frac{1}{q}+1=\frac{1}{p}+\frac{1}{r}.
		 	\end{equation} But this is true for any $r\ge 1$. Hence we are done. 
			\end{proof}

  As a corollary to the Theorem \ref{thm5.1} we have the following analogue of the result \ref{HLSH}.
\begin{cor} For  $ q = \frac{2n}{n-s},\,  0 < s< n $ we have the inequality
	\begin{equation}
	\label{HSLL}
	C_{n,s} \, \left(\int_{\mathbb{R}^n} |f(x)|^q dx\right)^{\frac{2}{q}}  \leq  \langle  H_s f, f\rangle  
	\end{equation}
 where $C_{n,s}$ is some constant depending only on $n$ and $s$.
\end{cor} 
\begin{proof}  Replacing $s$ by $s/2$ and putting $p=2$ in the above theorem we have 
	\begin{align}
	\|H_{-s/2}f\|^2_{q}\leq c_{n,s}\|f\|_2^2
	\end{align} where  $q=\frac{2n}{n-s}.$ Now in the above inequality substituting $f$ by $H_{s/2}f$ we have 
	  $$ \left(\int_{\mathbb{R}^n} |f(x)|^q dx\right)^{\frac{2}{q}}  \leq c_{n,s} \langle  H_{s/2} f, H_{s/2}f\rangle.$$ But in view of Stirling's formula for the gamma function we know that $H^2_{s/2}$ and $H_s$ differ by a bounded operator on $L^2(\mathbb{R}^n).$ Hence the result follows.
\end{proof} 
\begin{cor}[Hardy's inequality for $H_{s}$]
		Let $0<s<1$. Assume that $f\in L^2(\mathbb{R}^n)$ such that $H_{s}f\in L^2(\mathbb{R}^n)$.  Then we have 
	\[\langle H_sf,f\rangle_{L^2(\R^n)}\ge c_{n,s} \int_{\R^n}\frac{f(x)^2}{(1+|x|^2)^s}dx \]
\end{cor}
\begin{proof}
	Given $f\in L^2(\mathbb{R}^n)$, in view of Holder's inequality we have 
	 \begin{align}
	 \int_{\R^n}\frac{f(x)^2}{(1+|x|^2)^s}dx\leq A(n,s)\left(\int_{\mathbb{R}^n} |f(x)|^q dx\right)^{\frac{2}{q}}  
	 \end{align}
	 where $q=\frac{2n}{n-s}$, $A(n,s):=\left(\int_{\R^n}(1+|x|^2)^{-sq'}dx\right)^{1/q'} $ and $\frac{1}{q'}=1-\frac{2n-2s}{2n}=\frac{s}{n}.$ Hence the result follows from the previous corollary.   
\end{proof}
As a consequence of this we have a version of  Hardy's inequality for $L_s$:
\begin{cor}
	 	Let $0<s<1$. Assume that $f\in L^2(\gamma)$ such that $L_{s}f\in L^2(\gamma)$.  Then we have 
	 \[\langle L_sf,f\rangle_{L^2(\gamma)}\ge c_{n,s}     \int_{\R^n}\frac{f(x)^2}{(1+|x|^2)^s}d\gamma(x) \]
\end{cor}
\begin{proof}
	Let $f\in L^2(\gamma)$. Then it is easy to see that $g(x):=f(x)e^{-|x|^2/2}\in L^2(\mathbb{R}^n)$. By the above corollary we have 
	$$\langle H_sg,g\rangle_{L^2(\R^n)}\ge c_{n,s} \int_{\R^n}\frac{g(x)^2}{(1+|x|^2)^s}dx$$
	Also from the spectral decomposition we see that 
	$$H_sg(x )=H_{ s}(fe^{-|.|^2/2})(x)= e^{-|x|^2/2}L_{ s}f(x) $$
	which gives $\langle H_sg,g\rangle_{L^2(\R^n)}=\langle L_sf,f\rangle_{L^2(\gamma)}.$ Hence the result follows.
\end{proof}

\par 
\begin{rem}
In \cite{FLS}, Frank and Lieb proved that the constant appearing in the left hand side of the Hardy-Littlewood-Sobolev inequality \eqref{HLSH} for the sublaplacian on the Heisenberg group is sharp. It would be interesting to see the sharp constant in the analogous inequality \ref{HSLL} which we have proved for the Hermite operator.  
\end{rem}
 
\par
 \section*{Acknowledgments} 
 The first author  is supported by Int. Ph.D. scholarship from Indian Institute of Science. The second author is supported by C. V. Raman PDF, R(IA)CVR-PDF/2020/224 from Indian Institute of Science. And the third author is supported by  J. C. Bose Fellowship from D.S.T., Govt. of India.

\end{document}